\documentclass[leqno,12pt]{article} 
\usepackage{color}
\usepackage{ulem}
\usepackage{amsmath,amsfonts,amssymb,amsthm,amscd}
\usepackage{hyperref}
\usepackage{wrapfig}
\font\cmssl=cmss10 at 12 pt

\usepackage{amsmath}
\usepackage{amssymb}
\usepackage{amsthm}
\usepackage{pb-diagram}
\usepackage{ascmac}

\newcommand{\n}{\nabla}

\newtheorem{theorem}{Theorem}[section]
\newtheorem{proposition}[theorem]{Proposition}
\newtheorem{lemma}[theorem]{Lemma}
\newtheorem{corollary}[theorem]{Corollary}
\newtheorem{remark}[theorem]{Remark}
\newtheorem{example}[theorem]{Example}
\newtheorem{definition}[theorem]{Definition}

\begin{document}
\title{The H/Q-correspondence and a generalization of the supergravity c-map}
\author{%
Vicente Cort{\' e}s
and
Kazuyuki Hasegawa
}

\maketitle
\begin{abstract}
Given a hypercomplex manifold with a rotating vector field 
(and additional data), 
we construct a conical hypercomplex manifold. 
As a consequence,  
we associate a quaternionic manifold to a hypercomplex manifold of the same dimension 
with a rotating vector field. This is a generalization of the HK/QK-correspondence. 
As an application, we show that
a quaternionic manifold can be associated to a conical special complex manifold of half its dimension. 
Furthermore, a projective special complex manifold (with a canonical c-projective structure)  
associates with a quaternionic manifold. The latter is 
a generalization of the supergravity c-map. We do also show that 
the tangent bundle of any special complex manifold carries a canonical 
Ricci-flat hypercomplex structure, thereby generalizing the rigid c-map.\\


\end{abstract}

\tableofcontents
\section{Introduction}

The HK/QK-correspondence 
is a construction of a (pseudo-)quaternionic K{\"a}hler manifold from 
a  (pseudo-)hyper-K\"ahler manifold of the same dimension with a rotating vector field
(see Definition~\ref{rot:def} and \cite{Haydys,ACM,Hitchin,ACDM}).
This correspondence gives also the supergravity c-map, which associates 
a quaternionic K{\"a}hler manifold with a projective special K\"ahler manifold.
The supergravity c-map was introduced in theoretical physics \cite{FS}. 

The inverse construction of the HK/QK-correspondence is called the QK/HK\--cor\-res\-pon\-dence. 
It has been generalized to a Q/H-correspondence, a construction of hypercomplex manifolds from quaternionic manifolds \cite{CH}.
The purpose of this paper is to construct a quaternionic manifold from a hypercomplex manifold 
endowed with a rotating vector field and some extra data. 
We shall call this construction the hypercomplex/quaternionic-correspondence 
(H/Q-correspondence for short). 
We briefly explain how we obtain this correspondence. 
First we define the notion of a conical hypercomplex manifold (Definition \ref{conic_hyp_def}). 
Next we construct a conical hypercomplex manifold $\hat{M}$ for every hypercomplex manifold $M$ 
with a rotating vector field $Z$ (Theorem~\ref{conification}) and additional data: 
a two-form $\Theta$ on $M$, a U(1)-bundle over $M$ whose curvature satisfies 
(\ref{curvature_assumption}) and a function $f$ on $M$ such that $df=-\iota_{Z} \Theta$.
The manifold  $\hat{M}$ is endowed with a free action of the Lie algebra $\mathrm{Lie}\, \mathbb{H}^* \cong \mathbb{R}\oplus \mathfrak{su}(2)$ 
and its quotient space $\bar{M}$ 
carries a quaternionic structure, provided that the quotient map $\hat{M} \rightarrow \bar{M}$ is a submersion. The H/Q-correspondence is then defined as $M\mapsto \bar{M}$ 
(Theorems \ref{gene_hq_corresp} and \ref{swan_twist}). 
In addition, we show that $\bar{M}$ carries not only a quaternionic connection but also 
an (induced) affine quaternionic vector field (Proposition~\ref{affineX}). 
Note that we give an example of our H/Q-correspondence 
from a hypercomplex Hopf manifold, which does not admit any hyper-K\"ahler structure 
 (Example \ref{ex_hopf}). 
Therefore the H/Q-correspondence is a proper generalization of the HK/QK-correspondence.  
Examples like hypercomplex or quaternionic Hopf manifolds show that hypercomplex and quaternionic manifolds arise naturally beyond the 
context of hyper-K\"ahler and quaternionic K\"ahler geometry.
We refer to \cite{S, J1, J2} for the theory of quaternionic manifolds and constructions of such manifolds.

The rigid c-map \cite{CFG} allows to associate with a conical special K\"ahler manifold 
its cotangent bundle endowed with a hyper-K\"ahler structure with a rotating vector field \cite{ACM}. 
In the absence of a metric, we show that 
the tangent bundle of a special complex manifold carries a canonical hypercomplex structure and that  its Obata connection is Ricci flat (Theorem~\ref{ricci_flat}). 
In this way we establish a generalization of the rigid c-map which assigns 
a Ricci flat hypercomplex manifold to each special complex manifold. 
When the special complex manifold is conical, the resulting hypercomplex manifold
is shown to admit a canonical rotating vector field (Lemma~\ref{rotation_eq}). 
The notion of a (conical) special complex manifold was introduced in \cite{ACD}. 
It is a generalization of a (conical) special K\"ahler manifold.  
We give a local example which does not arise as a special K\"ahler manifold 
(Example \ref{ex_conic}).  
In addition, we find many (different) quaternionic structures on 
the tangent bundle of a conical special complex manifold in this example (Example 
\ref{ex_conic}), using a generalization of the supergravity c-map.

As an application of our H/Q-correspondence, we indeed generalize the supergravity c-map by 
associating a quaternionic manifold with every conical special complex manifold 
and therefore with every projective special 
complex manifold (using the extra data involved in the H/Q-correspondence), see Theorem~\ref{gene_sup_cmap}.
It is shown in Proposition~\ref{projective_str} that any 
projective special 
complex manifold possesses a canonical c-projective structure and    
in Theorem \ref{Weyl_bar} that its c-projective Weyl curvature is of type $(1,1)$.  
So our generalized supergravity c-map
can be formulated as 
associating a quaternionic manifold to 
a projective special 
complex manifold endowed with its canonical c-projective structure 
with c-projective Weyl curvature of type $(1,1)$. 
This addresses one of the questions raised in \cite{BC}, where 
a different construction of quaternionic manifolds from c-projective structures was obtained, 
compare Remark~\ref{BC:rem}. 

In the special case of the HK/QK-correspondence, the two-form 
 $\Theta$, which is part of the data 
 entering the H/Q-correspondence, 
 is the $Z$-invariant K{\"a}hler form $\omega_1$ in the hyper-K\"ahler-triple $(\omega_1,\omega_2,\omega_3)$. 
 However, in general, we have a freedom in the choice of $\Theta$ in the H/Q-correspondence  
(see Section~\ref{Examples:sec}). 
In particular we find two choices of $\Theta$ in Example \ref{ex_su_3} which yield 
different quaternionic structures on the resulting space. This shows that our H/Q-correspondence is not  
an inverse construction of the Q/H-correspondence without a further specification of $\Theta$. 
It is left for future studies  
to find a suitable choice of $\Theta$ which gives 
an inverse construction.

We summarize our constructions in this paper as the following commutative diagram. \\

{
\begin{wrapfigure}{l}{1cm}
\vspace{-0.7cm}
\tiny
\begin{align*}
\hspace{1cm}
\hat{M}&\mbox{: conical hypercomplex} \\
\bar{M}&\mbox{: quaternionic} \\
N&\mbox{: conical special complex} \\ 
\bar{N}&\mbox{: projective special complex}
\end{align*}
\end{wrapfigure}
\begin{eqnarray*}
\hspace{-7cm}
  \begin{diagram}
    \node[3]{(P,\eta)} 
    \arrow[1]{s,r}{\mathrm{U}(1)} \\
    \node[2]{ (N,J,\nabla,\xi)}     
    \arrow[1]{e,tb,..}{\mbox{\tiny rigid c-map.}}
                         {\mbox{\tiny Theorem \ref{ricci_flat} }}
    \arrow[1]{s,lr}{\mbox{\tiny Proposition \ref{projective_str}}}{p_{N}}
    \node[1] {(M=TN,f, \Theta)} 
    \arrow[1]{e,tb,..}{\mbox{\tiny conification}}
                         {\mbox{\tiny Theorem \ref{conification} } }
    \arrow[1]{se,tb,..}{\mbox{\tiny H/Q-corresp.}}
    {\mbox{\tiny Theorems \ref{gene_hq_corresp} and \ref{swan_twist}}} 
    \node[1]{\hat{M}={\cal C}_{P}(M)}
    \arrow[1]{s,r}{\hat{\pi}}\\
    \node[2]{(\bar{N},\bar{J},\mathcal{P}_{\bar{\nabla}^{\prime}})}
    \arrow[2]{e,tb,..}{\mbox{\tiny generalized supergravity c-map}}
                          {\mbox{\tiny Theorem \ref{gene_sup_cmap}}}
    \node[2]{\bar{M}=\hat{M}/\mathcal{D}}
  \end{diagram}
\end{eqnarray*}
}

\section{Preliminaries}
\setcounter{equation}{0}

Throughout this paper, 
all manifolds are assumed to be smooth and without boundary 
and maps are assumed to be smooth 
unless otherwise mentioned. 
The space of sections of a vector bundle $E\rightarrow M$ 
is denoted by $\Gamma(E)$. 

In this section we introduce hypercomplex and quaternionic structures
and derive some properties of conical hypercomplex manifolds.

We say that $M$ is a {\cmssl quaternionic 
manifold} with the {\cmssl quaternionic structure} $Q$ if 
$Q$ is a subbundle of $\mathrm{End}(T M)$ 
of rank $3$ which at every point $x\in M$ is spanned by endomorphisms  
$I_{1}$, $I_{2}$, $I_{3}\in \mathrm{End} (T_x M)$ satisfying
\begin{eqnarray}\label{quaternionic}
I_{1}^{2}=I_{2}^{2}=I_{3}^{2}=-\mathrm{id}, \,\, I_{1} I_{2}=-I_{2}I_{1}=I_{3}, 
\end{eqnarray}
and there exists a torsion-free connection $\nabla$ on $M$ 
such that $\nabla$ preserves $Q$, that is, 
$\nabla_X \Gamma(Q) \subset \Gamma(Q)$ for all $X\in\Gamma (TM)$. 
Such a torsion-free connection 
$\nabla$ is called a {\cmssl quaternionic connection} and 
the triplet $(I_{1}$, $I_{2}$, $I_{3})$ is called an 
{\cmssl admissible frame} of $Q$ at $x$. 
Note that we use the same letter $\nabla$ for the connection on ${\rm End}(T M)$ 
induced by $\nabla$. The dimension of the quaternionic manifold $M$ is denoted by $4n$. 
%
%
%

An {\cmssl almost hypercomplex manifold} is defined to be a manifold $M$ endowed 
with 3 almost complex structures 
$I_{1}$, $I_{2}$, $I_{3}$ satisfying the quaternionic relations (\ref{quaternionic}). 
If $I_{1}$, $I_{2}$, $I_{3}$ are integrable, then $M$ is called 
a {\cmssl hypercomplex manifold}. 
There exists a unique torsion-free connection on a hypercomplex manifold 
for which the hypercomplex structures are parallel. 
It is called the {\cmssl  Obata connection} \cite{O}. 
Obviously, hypercomplex manifolds are quaternionic manifolds with 
$Q=\langle I_{1},I_{2},I_{3} \rangle$.

\begin{definition}\label{conic_hyp_def}
We say that a hypercomplex manifold $(M,(I_{1},I_{2},I_{3}))$ with a vector field $V$ 
is {\cmssl conical} if 
$\nabla^{0} V=\mathrm{id}$ holds, where $\nabla^{0}$ is the Obata connection. 
The vector field $V$ is called the {\cmssl Euler} vector field.
\end{definition}

We state some lemmas for conical hypercomplex manifolds, which will be used later. 

\begin{lemma}\label{derivative_D}
Let $(M,(I_{1},I_{2},I_{3}),V)$ be a conical hypercomplex manifold. Then we have
$L_{V} I_{\alpha}=0$, $L_{I_{\alpha}V} I_{\alpha}=0$ 
for $\alpha \in \{ 1,2,3\} $ 
and $L_{I_{\alpha}V}I_{\beta}=-2 I_{\gamma}$ 
for any cyclic permutation $(\alpha,\beta,\gamma)$.
\end{lemma}

\begin{proof} The formulas follow immediately from $L_V= \nabla^0_V-\nabla^0V= \nabla^0_V-\mathrm{id}$ and $L_{I_\alpha V} = \nabla^0_{I_\alpha V} - I_\alpha$. 
\end{proof}

For a connection $\nabla$ and $X \in \Gamma(TM)$, we define
\begin{eqnarray}\label{affine}
(L_{X}  \nabla)_{Y} Z
:= L_{X} (\nabla_{Y}Z)-\nabla_{L_{X}Y}Z-\nabla_{Y}(L_{X}Z), 
\end{eqnarray}
where $Y$, $Z \in \Gamma(TM)$. Note that 
$L_{X}  \nabla$ is a tensor. 

\begin{lemma}\label{derivative_Obata}
Let $(M,(I_{1},I_{2},I_{3}),V)$ be a conical hypercomplex manifold. Then we have
$L_{V} \nabla^{0}=0$ and $L_{I_{\alpha}V} \nabla^{0}=0$. 
\end{lemma}

\begin{proof}
By Lemma \ref{derivative_D}, $V$ and $I_{\alpha}V$ are quaternionic vector fields, namely 
$L_{V} \Gamma(Q) \subset \Gamma(Q)$ and 
$L_{I_{\alpha}V} \Gamma(Q) \subset \Gamma(Q)$, where $Q=\langle I_{1},I_{2},I_{3} \rangle$. 
By \cite[Proposition 4.2]{CH}, it is enough to check $Ric^{\nabla^{0}}(V, \, \cdot \,)=0$ 
and $Ric^{\nabla^{0}}(I_{\alpha}V, \, \cdot \,)=0$. We have
\[ Ric^{\nabla^{0}}(V,Y)=-Ric^{\nabla^{0}}(Y,V)=-\mathrm{Tr} \, R^{\nabla^{0}}(\, \cdot \,,Y)V=0. \]
Here we used the skew-symmetry of the Ricci tensor of the Obata connection. 
It follows that also $Ric^{\nabla^{0}}(I_{\alpha}V, \, \cdot \,)=- Ric^{\nabla^{0}}(V, I_{\alpha}\, \cdot \,)=0$, by the hermitian property of the 
Ricci tensor of the Obata connection. 
\end{proof}

Alternatively we could have used Lemma \ref{derivative_D} and the explicit form of the Obata connection
to check $L_{I_{\alpha}V} \nabla^{0}=0$. Note that $L_{V} \nabla^{0}=0$ follows from the uniqueness of the Obata connection, since the 
vector field $V$ preserves the hypercomplex structure.

\begin{example}[The Swann bundle]\label{ex_sw}
{\rm 
The principal $\mathbb{R}^{>0} \times \mathrm{SO(3)}$ bundle over a quaternionic manifold, 
whose fibers consist of all volume elements and admissible frames at each point, possesses  
a hypercomplex structure (see \cite{PPS, CH}). It is conical 
and is called the Swann bundle. The   
fundamental vector field generated by $c(\neq 0) \in T_1\mathbb{R}^{>0}=\mathbb{R}$ 
is the Euler vector field, as can be easily checked from the explicit representation of the Obata connection (see \cite{AM} for example). 
In the notation of \cite{CH} with $\varepsilon=-1$ and $c=-4(n+1)$, 
a basis of fundamental vector fields for the principal action is given by the vector fields 
$V=Z_0$ and $Z_\alpha=-I_\alpha Z_0$ with non-trivial commutators $[Z_\alpha, Z_\beta ] = -2 Z_\gamma$ and Lie derivatives $L_{Z_\alpha}I_\beta = -2 I_\gamma$ for any cyclic permutation of $\{ 1,2,3\}$, where we 
have denoted by $(I_1,I_2,I_3)$ the hypercomplex structure of the Swann bundle. 
Specializing to the Swann bundle $\mathbb{H}^*/\{ \pm 1\}$ 
of a point, we see that $Z_0$ corresponds to $1$ and $(Z_1,Z_2,Z_3)$ to $(i,j,k)$ in $T_1(\mathbb{H}^*/\{ \pm 1\})= T_1\mathbb{H}=\mathbb{H}$.
}\end{example}

\begin{lemma}
On any conical hypercomplex manifold $(M,(I_{1},I_{2},I_{3}),V)$, 
the distribution $\mathcal D:=\langle V, I_{1}V,I_{2}V, I_{3}V \rangle$ 
on $\{ x \in M \mid V_{x} \neq 0 \}$ 
is integrable.
\end{lemma}

\begin{proof} This follows from Lemma \ref{derivative_D}.
\end{proof}

\section{Conification of hypercomplex manifolds}
\label{conif:sec}
\setcounter{equation}{0}
The main result of this section is a construction of conical 
hypercomplex manifolds $\hat M$ of dimension $\dim \hat M = \dim M + 4$ from hypercomplex manifolds $M$ with a rotating 
vector field.

Let $M$ be a hypercomplex manifold of dimension $4n$ with a hypercomplex structure 
$H=(I_{1},I_{2},I_{3})$.

\begin{definition} \label{rot:def}A vector field $Z$ on a hypercomplex manifold $(M, (I_1,I_2,I_3))$ is called 
{\cmssl rotating} if $L_{Z} I_{1}=0$ and 
$L_{Z}I_{2}=-2 I_{3}$. 
\end{definition}

Note that if $Z$ is rotating, then $L_{Z}I_{3}=2 I_{2}$.
In this section we will essentially show that by choosing 
a (local) primitive of the one-form $\iota_Z\Theta$ 
we can construct a conical hypercomplex manifold $(\hat{M},\hat{H},V)$ 
for a hypercomplex manifold $(M,H)$ with a rotating vector field $Z$ and a closed two-form $\Theta$ such that $L_{Z} \Theta=0$. 

Let $f$ be a smooth function on $M$ such that $df = -\iota_{Z} \Theta$ and 
$f_{1}:=f-(1/2)\Theta(Z,I_{1}Z)$ is nowhere vanishing. 
Consider a principal $\mathrm{U}(1)$-bundle $\pi:P \to M$ with a connection form $\eta$ 
whose curvature form is 
\begin{align}\label{curvature_assumption}
d \eta
   =\pi^{\ast} \left( \Theta-\frac{1}{2} d ((\iota_{Z} \Theta) \circ I_{1}) \right). 
\end{align}
Since the curvature $d\eta$ is a basic form, we will usually identify it with its projection  $\Theta-\frac{1}{2} d ((\iota_{Z} \Theta) \circ I_{1})$ 
on $M$. With this 
understood we have the following lemma, which follows immediately from the definition of $f_1$.
\begin{lemma} \label{deta:lemma}
$df_1 = -\iota_Zd\eta$.
\end{lemma}

Define a vector field $Z_{1}$ on $P$ by $Z_{1}=Z^{h_{\eta}}
+(\pi^{\ast }f_{1}) X_{P}$,
where $Z^{h_{\eta}}$ is the $\eta$-horizontal lift and  
$X_{P}$ is the fundamental vector field such that $\eta(X_{P})=1$. 
We will write $f_{1}$ for $\pi^{\ast}f_{1}$.

\begin{remark}
Note that $[X_{P},Z_{1}]=0$. 
Therefore if $Z_{1}$ generates a $\mathrm{U}(1)$-action on $P$, then its action commutes with the principal action
of $\pi:P \to M$. 
\end{remark}

Set $\tilde{M}={\mathbb H}^{\ast} \times  P$. 
Let $(e_{0}^{R}, e_{1}^{R}, e_{2}^{R}, e_{3}^{R})$ 
(resp. $(e_{0}^{L}, e_{1}^{L}, e_{2}^{L}, e_{3}^{L})$) 
be the right-invariant (resp. the left-invariant) 
frame of ${\mathbb H}^{\ast}$
which coincides with $(1,i,j,k)$ at $1 \in {\mathbb H}^{\ast}$. 
Note that $[e_{1}^{R}, e_{2}^{R}]=-2e_{3}^{R}$. 
We will use the same letter for vectors or vector fields canonically lifted to the product 
$\tilde{M}={\mathbb H}^{\ast} \times  P$ 
as for those on the factors ${\mathbb H}^{\ast}$ and $P$. 
Set 
\[ V_{1}:=e_{1}^{L}-Z_{1}. \] 
We denote the space of integral curves of $V_{1}$ by $\hat{M}$.  
We assume that the quotient map $\tilde{\pi}:\tilde{M} \to \hat{M}$ 
is a submersion. Note that ``submersion" requires that the quotient space $\hat{M}$ is smooth.

\begin{lemma}\label{derivative_V1}
We assume that the equation (\ref{curvature_assumption}) holds. 
If $L_{Z} I_{1}=0$ and $L_{Z} \Theta=0$, 
we have
\[ L_{V_{1}} Y^{h_{\eta}}=-[Z,Y]^{h_{\eta}} \]
for all $Y \in \Gamma(TM)$.
\end{lemma}
\begin{proof}
\begin{align*}
-L_{V_{1}} Y^{h_{\eta}}{} 
=&-[e_{1}^{L}-Z_{1},Y^{h_{\eta}}{} ] = [Z_{1},Y^{h_{\eta}}{}] \\
=&[Z^{h_{\eta}}{},Y^{h_{\eta}}{}]+[f_{1}X_{P},Y^{h_{\eta}}{}] = [Z^{h_{\eta}}{},Y^{h_{\eta}}{}]-(Y^{h_{\eta}}{}f_{1})X_{P}\\
=&[Z,Y]^{h_{\eta}}{}+\eta([Z^{h_{\eta}}{},Y^{h_{\eta}}{}])X_{P}-(Y^{h_{\eta}}{}f_{1})X_{P}\\
=&[Z,Y]^{h_{\eta}}{}-d \eta(Z,Y)X_{P}- (Yf_{1})X_{P}\\
=&[Z,Y]^{h_{\eta}}{},
\end{align*}
where we have used Lemma \ref{deta:lemma}.
\end{proof}

Note that
\begin{align*}
T_{(z,p)} \tilde{M} 
&\cong T_{z} {\mathbb H}^{\ast} \oplus T_{p}P 
= \langle e_{0}^{R}, e_{1}^{R}, e_{2}^{R}, e_{3}^{R} \rangle_{z} \oplus  \langle X_{P}\rangle_{p} 
      \oplus {\rm Ker} \, \eta_{p} \\
&= \langle V_{1} \rangle_{(z,p)} \oplus \langle e_{0}^{R}, e_{1}^{R}, e_{2}^{R}, e_{3}^{R} \rangle_{z} 
      \oplus {\rm Ker} \, \eta_{p} 
\end{align*}
for $(z,p) \in {\mathbb H}^{\ast} \times  P$. 
We define three endomorphisms fields $\tilde{I}_{1},\tilde{I}_{2},\tilde{I}_{3}$ on $\tilde{M}$ 
of $\mathrm{rank}\, \tilde{I}_{\alpha}=4n+4$ $(\alpha=1,2,3)$ as follows: 
\begin{align*}
&\tilde{I}_{\alpha} V_{1} = 0, \,
\tilde{I}_{\alpha} e_{0}^{R} =e_{\alpha}^{R}, \,
\tilde{I}_{\alpha} e_{\alpha}^{R} =-e_{0}^{R}, \,
\tilde{I}_{\alpha} e_{\beta}^{R} = e_{\gamma}^{R}, \, 
\tilde{I}_{\alpha} e_{\gamma}^{R} = -e_{\beta}^{R}, \\
&(\tilde{I}_{\alpha})_{(z,p)} ((Y^{h_{\eta}}{})_{(z,p)})
=((I^{\prime}_{\alpha})_{\pi(p)}(\pi_{\ast}Y))^{h_{\eta}}{}_{(z,p)} 
\end{align*}
for $Y \in T_pM$. 
Here $I^{\prime}_{\alpha}$ is defined by 
\begin{align}\label{I_prime}
I^{\prime}_{\alpha}=\sum_{\beta=1}^{3} A_{\alpha \beta} I_{\beta}, 
\end{align}
where 
$A=(A_{\alpha \beta}) \in \mathrm{SO}(3)$ is the representation matrix of 
$\mathrm{Ad}_{z} | _{ \mathrm{Im} \mathbb H}$ with respect to the basis $(i,j,k)$. 
Note that $\mathrm{Ker} \, \tilde{I}_{\alpha}=\langle V_{1} \rangle$, 
$\mathrm{Im} \, \tilde{I}_{\alpha} =T\mathbb{H}^*
\oplus \mathrm{Ker}\, \eta$ $(\alpha=1,2,3)$ and that $\tilde{I}_{1},\tilde{I}_{2},\tilde{I}_{3}$ satisfy the quaternionic relations 
on 
$T\mathbb{H}^*
\oplus {\rm Ker} \, \eta$. 

\begin{lemma} \label{e0R:lemma}$L_{e_0^R}\tilde{I}_\alpha=0$.
\end{lemma}

\begin{proof}The flow $\varphi_t : (z,p)\mapsto (e^tz,p)$ of $e_0^R$ preserves the 
decomposition $\tilde{M} = \mathbb{H}^*\times P$ and acts trivially 
on the second factor. In particular, it preserves the distribution 
$\mathrm{Ker}\, \eta$. The action on the first factor is tri-holomorphic with respect 
to the (standard) hyper-complex structure induced by $(\tilde{I}_\alpha)$ on $\mathbb{H}^*$.  
Since $\mathrm{Ad}_z=\mathrm{Ad}_{rz}$ for all 
$r>0$, we also see that $\varphi_t$ preserves the tensors $\tilde{I}_\alpha|_{\mathrm{Ker}\, \eta}$.
\end{proof}

\begin{lemma}\label{invariant_partial}
If $Z$ is rotating and $L_{Z} \Theta=0$, then 
we have $L_{V_{1}} \tilde{I}_{\alpha}=0$.
\begin{proof}
By the definition of $\tilde{I}_{\alpha}$, it is easy to obtain 
$(L_{V_{1}} \tilde{I}_{\alpha})V_{1} =0$ and 
$(L_{V_{1}} \tilde{I}_{\alpha})e_{\delta}^{R} =0$ 
$(\delta=0,\ldots ,3)$. Moreover, by Lemma \ref{derivative_V1}, we have
\begin{align*}
&(L_{V_{1}} \tilde{I}_{\alpha})_{(z,p)}(Y^{h_{\eta}}{}) \\
=&[V_{1}, \tilde{I}_{\alpha} Y^{h_{\eta}}{}]_{(z,p)}-\tilde{I}_{\alpha}[V_{1},Y^{h_{\eta}}{}]_{(z,p)} \\
=&[e_{1}^{L},\tilde{I}_{\alpha} Y^{h_{\eta}}{} ]_{(z,p)}-[Z_{1},\tilde{I}_{\alpha} Y^{h_{\eta}}{}]_{(z,p)}
   +\tilde{I}_{\alpha} [Z,Y]^{h_{\eta}}{}_{(z,p)}\\
=& [e_{1}^{L},\tilde{I}_{\alpha} Y^{h_{\eta}}{} ]_{(z,p)}-[Z^h,\tilde{I}_{\alpha} Y^{h_{\eta}}{}]_{(z,p)} -[f_1X_P, \tilde{I}_{\alpha} Y^{h_{\eta}}{}]_{(z,p)}+
(I_{\alpha}' [Z,Y])^{h_{\eta}}{}_{(z,p)}\\
=&[e_{1}^{L},\tilde{I}_{\alpha} Y^{h_{\eta}}{} ]_{(z,p)}-((L_{Z} I^{\prime}_{\alpha})Y)^{h_{\eta}}{}_{(z,p)},
\end{align*}
where we have used that $[Z^h_{\eta},\tilde{I}_{\alpha} Y^{h_{\eta}}{}]
+[f_1X_P, \tilde{I}_{\alpha} Y^{h_{\eta}}{}]=[Z,I_{\alpha}'Y]^{h_{\eta}}+\eta ([Z,I_{\alpha}'Y])X_P
-(I_{\alpha}'Y)(f_1)X_P=[Z,I_{\alpha}'Y]^{h_{\eta}}$ at the point $(z,p)$, by Lemma \ref{deta:lemma}.
Taking the flow $\varphi_{t}$ generated by $e_{1}^{L}$, we have 
\[ [e_{1}^{L},\tilde{I}_{\alpha} Y^{h_{\eta}}{} ]_{(z,p)}
= \sum_{\beta=1}^{3} \left(\left. \frac{d}{dt}\right|_{t=0} A_{\alpha \beta}(t) \right)(I_{\beta} Y)^{h_{\eta}}{}_{(z,p)}, \]
where 
\[ A(t)=(A_{\alpha \beta}(t))
   =\left( \begin{array}{ccc}
                    1 &  0 &  0  \\
                    0 &  \cos 2t & \sin 2t \\
                    0 &  -\sin 2t &  \cos 2t \\
                    \end{array} 
             \right) \in \mathrm{SO}(3),
\]
is the matrix associated with $\varphi_{t}(z)$. 
On the other hand, we see that
\begin{align*}
L_{Z} I_{1}^{\prime} &=-2A_{12}I_{3}+2 A_{13}I_{2}, \\
L_{Z} I_{2}^{\prime} &=-2A_{22}I_{3}+2 A_{23}I_{2}, \\
L_{Z} I_{3}^{\prime} &=-2A_{32}I_{3}+2 A_{33}I_{2} 
\end{align*}
and hence
\[ L_{Z} (I_{1}^{\prime},I_{2}^{\prime},I_{3}^{\prime})
=(L_{Z} I_{1}^{\prime},L_{Z}I_{2}^{\prime},L_{Z}I_{3}^{\prime})
=(I_{1},I_{2},I_{3}) \left( \frac{d}{dt} A(t) \right). 
\]
Therefore we have $(L_{V_{1}} \tilde{I}_{\alpha})_{(z,p)}(Y^{h_{\eta}}{})=0$. 
\end{proof}
\end{lemma}

By Lemma \ref{invariant_partial}, we can define 
an almost hypercomplex structure $(\hat{I}_{1},\hat{I}_{2},\hat{I}_{3})$ on $\hat{M}$
satisfying $\tilde{\pi}_{\ast} \circ \tilde{I}_{\alpha}=\hat{I}_{\alpha} \circ \tilde{\pi}_{\ast}$.

\begin{lemma}\label{integrable}
The almost hypercomplex structure $\hat{H}=(\hat{I}_{1},\hat{I}_{2},\hat{I}_{3})$ is integrable, that is, 
$(\hat{M},\hat{H})$ is a hypercomplex manifold.
\begin{proof}
Let  $\tilde{X}$ and $\tilde{Y}$ be projectable vector fields on the total space of the submersion $\tilde{\pi}:\tilde{M}\rightarrow \hat{M}$ and denote by 
$X=\tilde{\pi}_{\ast} \tilde{X}$, $Y=\tilde{\pi}_{\ast} \tilde{Y}$ their projections. Then 
we have $\tilde{\pi}_{\ast}(N^{\tilde{I}_{\alpha}}(\tilde{X},\tilde{Y}))=N^{\hat{I}_{\alpha}}(X,Y)$,
where $N^{\tilde{I}_{\alpha}}$ and $N^{\hat{I}_{\alpha}}$ are the Nijenhuis tensors of 
$\tilde{I}_{\alpha}$ and $\hat{I}_{\alpha}$, respectively. 
Using that $\tilde{I}_\alpha V_1=0$ and $L_{V_1}\tilde{I}_\alpha=0$ (Lemma \ref{invariant_partial}) we see that $N^{\tilde{I}_{\alpha}}(V_{1},\cdot )=0$. 
Since $N^{\tilde{I}_{\alpha}}$ and $N^{\hat{I}_{\alpha}}$ are tensors, 
it is sufficient to show that 
the horizontal component of $N^{\tilde{I}_{\alpha}}(A,B)$ vanishes
for sections $A$ and $B$ of 
$\langle e_{0}^{R}, e_{1}^{R}, e_{2}^{R}, e_{3}^{R} \rangle
\oplus \mathrm{Ker}\, \eta$. 
It is easy to see that
$N^{\tilde{I}_{\alpha}}(e^{R}_{a},e^{R}_{b})=0$ and $N^{\tilde{I}_{\alpha}}(e^{R}_{a},X^{h_{\eta}}{})=0$, for all $a,b\in \{0,\ldots,3\}$. 
So we only need to show 
that the horizontal component of $N^{\tilde{I}_{\alpha}}(X^{h_{\eta}}{},Y^{h_{\eta}}{})$ vanishes, 
i.e.\ the component in $\langle e_{0}^{R}, e_{1}^{R}, e_{2}^{R}, e_{3}^{R} \rangle \
\oplus \mathrm{Ker}\, \eta$. It is given by
\[ \left( [X,Y]+I_{\alpha}^{\prime}[X,I_{\alpha}^{\prime}Y] 
+ I_{\alpha}^{\prime}[I_{\alpha}^{\prime}X, Y]
- [I_{\alpha}^{\prime}X, I_{\alpha}^{\prime}Y] \right)^{h_{\eta}}{}
=0,\]
since $(I_{1}^{\prime},I_{2}^{\prime},I_{3}^{\prime})$ is a hypercomplex structure 
on $M$, for every $z\in \mathbb{H}^*$.
\end{proof}
\end{lemma}

Since $L_{V_{1}} e_{0}^{R}=0$, 
we can define a vector field 
$V= \tilde{\pi}_{\ast}{e}_{0}^{R}$ on $\hat{M}$. 
Let $\hat{\nabla}^{0}$ be the Obata connection with respect to $\hat{H}$.

\begin{lemma}\label{Euler}
We have $\hat{\nabla}^{0} V={\rm id}$. 
\end{lemma}
\begin{proof}
Using the explicit representation of the Obata connection (see \cite{AM} for example) 
and Lemma \ref{e0R:lemma},
we have
\begin{align*}
12 (\hat{\nabla}^{0}_{\tilde{\pi}_{\ast}Y} \tilde{\pi}_{\ast} e_{0}^{R})
=&\tilde{\pi}_{\ast} \left( \sum_{(\alpha,\beta,\gamma)} (\tilde{I}_{\alpha} [\tilde{I}_{\beta} Y,e_{\gamma}^{R}] +\tilde{I}_{\alpha} [e_{\beta}^{R}, \tilde{I}_{\gamma} Y] )
+2 \sum_{\alpha =1}^{3} \tilde{I}_{\alpha}[e_{\alpha}^{R},Y] \right), 
\end{align*}
where $(\alpha,\beta,\gamma)$ indicates sum over cyclic permutations of $(1,2,3)$ 
and $Y$ is a projectable vector field on 
$\tilde{M}$ commuting with $e_0^R$. Evaluating the expression on $Y=e_{a}^{R}$ and $Y=U^{h_{\eta}}{}$, we obtain $12\tilde{\pi}_\ast Y$. 
\end{proof}

As a consequence, by Lemmas \ref{integrable} and \ref{Euler}, 
we can conclude

\begin{theorem}[Conification]\label{conification}
Let $M$ be a hypercomplex manifold with a hypercomplex structure 
$H=(I_{1},I_{2},I_{3})$, a closed two-form $\Theta$ and 
a rotating vector field $Z$ such that $L_{Z} \Theta=0$. 
Let $f$ be a smooth function on $M$ such that $df = -\iota_{Z} \Theta$ and 
assume $f_{1}:=f-(1/2)\Theta(Z,I_{1}Z)$ does nowhere vanish. 
Consider a principal $\mathrm{U}(1)$-bundle $\pi:P \to M$ with a connection form $\eta$ 
whose curvature form is 
\[ d \eta
   =\pi^{\ast} \left( \Theta-
   \frac{1}{2} d ((\iota_{Z} \Theta) \circ I_{1}) \right). \] 
If the quotient map $\tilde{\pi}:\tilde{M} \to \hat{M}$ 
is a submersion, then $(\hat{M},\hat{H})$ is a conical hypercomplex manifold with 
the Euler vector field $V=\tilde{\pi}_{\ast}{e}_{0}^{R}$. 
\end{theorem}

\begin{remark}\label{rem1}
{\rm 
The assumption that $\tilde{\pi}$ is a submersion 
is always satisfied locally by considering local $1$-parameter subgroup generated by $V_{1}$, 
since the vector field $V_{1}$ has no zeros. 
Note that ``submersion" requires that the quotient space is a smooth manifold. 
}
\end{remark}

We say that $(\hat{M},\hat{H},V)$ is the {\cmssl conification} of $(M,H,Z,f,\Theta)$ associated with 
$(P,\eta)$ 
and denote it by $(\hat{M},\hat{H},V)={\cal C}_{(P,\eta)}(M,H,Z,f,\Theta)$
(or simply $\hat{M}={\cal C}_{P}(M)$ if there is no confusion).

\section{The hypercomplex/quaternionic-correspondence}
\setcounter{equation}{0}
Building on the conification construction of the last section
we will now construct a quaternionic manifold $\bar M$ of dimension $\dim \bar M = \dim M$
from a hypercomplex manifold $M$ with rotating vector field. The resulting quaternionic manifold
is endowed with a torsion-free quaternionic connection and an affine quaternionic vector field $X$.
 
The space of leaves of  the integrable distribution 
$\mathcal D:=\langle V, \hat{I}_{1}V, \hat{I}_{2}V, \hat{I}_{3}V \rangle$ on $\hat{M}$
is denoted by $\bar{M}$. We shall show that 
$\bar{M}=\mathcal{C}_{P}(M) / \mathcal D$ is a quaternionic manifold, 
which is the main theorem of this paper. 
In addition, we show that $\bar{M}$ has a natural quaternionic connection 
$\bar{\nabla}$ and 
an affine quaternionic vector field $X$ induced 
from the fundamental vector field $X_P$ of $P\rightarrow M$.

Using Theorem \ref{conification} and a similar argument as 
in \cite[Theorem 2.1]{PPS}, we prove Theorem~\ref{gene_hq_corresp}.

\begin{theorem}[H/Q-correspondence]\label{gene_hq_corresp}
Let $M$ be a hypercomplex manifold with a hypercomplex structure 
$H=(I_{1},I_{2},I_{3})$, a closed two-form $\Theta$ and a rotating vector field $Z$ 
such that $L_{Z} \Theta=0$. 
Let $f$ be a smooth function on $M$ such that $df = -\iota_{Z} \Theta$ and assume that 
$f_{1}:=f-(1/2)\Theta(Z,I_{1}Z)$ does nowhere vanish. 
Consider a principal $\mathrm{U}(1)$-bundle $\pi:P \to M$ with a connection form $\eta$ 
whose curvature form is 
\[ d \eta
   =\pi^{\ast} \left( \Theta-
   \frac{1}{2} d ((\iota_{Z} \Theta) \circ I_{1}) \right). \] 
If both quotient maps $\tilde{\pi}:\tilde{M} \to \hat{M}$ and 
$\hat{\pi}:\hat{M} \to \bar{M}$ defined above 
are submersions, then there exists an induced quaternionic structure $\bar{Q}$ on $\bar{M}$.
\end{theorem}

\begin{proof}
As we proved in Theorem \ref{conification}, $\hat{M}=\mathcal{C}_{P}(M)$ 
is a conical hypercomplex manifold with 
the hypercomplex structure $\hat{H}=(\hat{I}_{1},\hat{I}_{2},\hat{I}_{3})$. 
Let $\varphi=\sum_{a=0}^{3} \varphi_{a}i_{a}$  
($(i_{0},i_{1},i_{2},i_{3})=(1,i,j,k)$) be the right-invariant Maurer-Cartan form on 
$\mathbb{H}^{\ast}$ and extend it with the same letter to $\tilde{M}$ as $\varphi|_{TP}=0$.   
Set $\tilde{\theta}_{0}=\varphi_{0}$.  
Since $L_{V_{1}} \tilde{\theta}_{0}=0$, we can define 
the one-form $\hat{\theta}_{0}$ on $\hat{M}$ 
such that $\tilde{\theta}_{0}=\tilde{\pi}^{\ast} \hat{\theta}_{0}$. 
We define 
$\hat{\theta}^{\prime}=\hat{\theta}_{0}+\sum_{\alpha=1}^3 (\hat{\theta}_{0} \circ \hat{I}_{\alpha})i_{\alpha}$ 
and take the Euler vector field $V$ on $\hat{M}$ as in Theorem \ref{conification}. 
Here define an $\hat{I}_{\alpha}$-invariant distribution 
\[ \mathcal{\hat{H}}:=\mathrm{Ker} \, \hat{\theta}^{\prime}. \]
It holds that $T \hat{M}=\mathcal{D} \oplus \mathcal{\hat{H}}$. 
Since $L_{V} \hat{\theta}^{\prime}=0$ and 
$L_{\hat{I}_{\alpha}V} \hat{\theta}^{\prime}
=2 (\hat{\theta}_{0} \circ \hat{I}_{\beta}) i_{\gamma}
-2 (\hat{\theta}_{0} \circ \hat{I}_{\gamma}) i_{\beta}$ 
for any cyclic permutation $(\alpha,\beta,\gamma)$ 
(these are checked by straightforward calculations), 
the distribution $\mathcal{\hat{H}}$ is invariant along leaves of $\mathcal{D}$.
Since $\hat{\pi}$ is a submersion, there exist a neighborhood $\mathcal{U} \subset \bar{M}$ of 
$x \in \bar{M}$ and a section $s:\mathcal{U} \to \hat{M}$.
Then we can define 
\[ \bar{I}_{\alpha}(Y)
:=\hat{\pi}_{\ast}(\hat{I}_{\alpha}(Y^{h_{\hat{\theta}^{\prime}}}_{s(y)})) \]
for $y \in \mathcal{U}$, where $Y \in T_{y} \bar{M}$ and 
$Y^{h_{\hat{\theta}^{\prime}}}$ 
is the $\hat{\theta}^{\prime}$-horizontal lift of $Y$ 
with respect to $\mathcal{\hat{H}}$. 
Although each $\bar{I}_{\alpha}$ depends on the sections, 
the subbundle $\bar{Q}=\langle \bar{I}_{1}, \bar{I}_{2}, \bar{I}_{3} \rangle 
\subset \mathrm{End}(T \bar{M})$ is independent of the section by Lemma \ref{derivative_D}. 
This means that $(\bar{M},\bar{Q})$ is an almost quaternionic manifold. 

Next we show that there exists a torsion-free connection which preserves $\bar{Q}$. 
We define a connection $\bar{\nabla}$ on $\bar{M}$ by 
\begin{align}\label{q_conn_barM}
\bar{\nabla}_{Y} W &= \hat{\pi}_{\ast} (\hat{\nabla}^{0}_{Y^{h_{\hat{\theta}^{\prime}}}} W^{h_{\hat{\theta}^{\prime}}}),\quad Y, W \in \Gamma(T\bar{M}), 
\end{align}
where $\hat{\nabla}^{0}$ is the Obata connection of $\hat{M}$. 
Note that $\bar{\nabla}$ is well-defined by Lemma \ref{derivative_Obata}. 
Since the Obata connection is torsion-free, then so is $\bar{\nabla}$. 
To show that $\bar{\nabla}$ preserves $\bar{Q}$, we consider $I \in \Gamma(\bar{Q})$. Then $(IW)^{h_{\hat{\theta}^{\prime}}}
=\sum_{\alpha=1}^{3} a_{\alpha} \hat{I}_{\alpha} W^{h_{\hat{\theta}^{\prime}}}$ 
for some functions $a_\alpha$ with $\sum_{\alpha=1}^{3} a_{\alpha}^{2}=1$, 
which implies 
\[
(\bar{\nabla}_{Y} I)W=\hat{\pi}_{\ast}
(\sum_{\alpha=1}^{3} (Y^{h_{\hat{\theta}^{\prime}}} a_{\alpha})\hat{I}_{\alpha}W^{h_{\hat{\theta}^{\prime}}} ), 
\]
showing that $\bar{\nabla}$ preserves $\bar{Q}$. 
Therefore $(\bar{M},\bar{Q})$ is a quaternionic manifold. 
\end{proof}

\begin{remark}\label{rem2}
{\rm 
The assumption that $\hat{\pi}$ is a submersion 
is always satisfied locally. 
}
\end{remark}

Next we shall show that our construction induces 
a vector field $X$ which is an affine quaternionic vector field of $(\bar{M},\bar{Q},\bar{\nabla})$, where $\bar{\nabla}$ is given by (\ref{q_conn_barM}).

\begin{lemma}\label{invariance_by_X1}
We have $L_{V_{1}} X_{P}=0$ and $L_{X_{P}} \tilde{I}_{\alpha}=0$.
\end{lemma}

\begin{proof}
The first equation can be checked by a straightforward calculation. 
The second follows from $[X_{P},\tilde{I}_{\alpha}Y^{h_{\eta}}]
=[X_P,(I_\alpha'Y)^{h_{\eta}}]=0$.
\end{proof}

By Lemma \ref{invariance_by_X1}, we can define a vector field $\widehat{X_{P}}:=\tilde{\pi}_{\ast}X_{P}$
on $\hat{M}$. Moreover $\widehat{X_{P}}$ satisfies the following.

\begin{lemma}\label{invariance_by_X2}
We have $L_{\widehat{X_{P}}} \hat{I}_{\alpha}=0$, in addition, $L_{\widehat{X_{P}}} \hat{\nabla}^{0}=0$. 
\end{lemma}

\begin{proof}
The first claim follows from
Lemma \ref{invariance_by_X1}, as 
$(L_{\widehat{X_{P}}} \hat{I}_{\alpha}) \circ \tilde{\pi}_{\ast}
=\tilde{\pi}_{\ast} \circ (L_{X_{P}} \tilde{I}_{\alpha}$). 
Since the Obata connection is uniquely determined by the hypercomplex structure, 
we have $L_{\widehat{X_{P}}} \hat{\nabla}^{0}=0$ by the invariance 
of the hypercomplex structure $(\hat{I}_{1},\hat{I}_{2},\hat{I}_{3})$ 
under $\widehat{X_{P}}$. 
\end{proof}

The next two lemmas follow respectively 
from $[e_a^R,X_P]=0$ and $L_{X_{P}} \tilde{\theta}_{0}=0$ by projection.

\begin{lemma}\label{invariance_by_X3}
We have $L_{V} \widehat{X_{P}}=0$ and $L_{\hat{I}_{\alpha}V} \widehat{X_{P}}=0$. 
\end{lemma}
\begin{lemma}\label{invariance_by_X4}
We have $L_{\widehat{X_{P}}} \hat{\theta}_{0}=0$ on $\hat{M}$. 
\end{lemma}

Lemma \ref{invariance_by_X3} allows us to define a vector field 
$X := \hat{\pi}_{\ast}\widehat{X_{P}}$ on $\bar{M}$. 

\begin{proposition}\label{affineX}
Let $(\bar{M},\bar{Q})$ be a quaternionic manifold obtained from 
a hypercomplex manifold $M$ satisfying the assumptions in Theorem \ref{gene_hq_corresp}
and $\bar{\nabla}$ the quaternionic connection defined by 
{\rm (\ref{q_conn_barM})}. 
The vector field $X$ is an affine quaternionic vector field of 
$(\bar{M},\bar{Q},\bar{\nabla})$, 
that is, satisfies $L_{X} \Gamma(\bar{Q}) \subset \Gamma(\bar{Q})$ and 
$L_{X} \bar{\nabla}=0$. 
\end{proposition}

\begin{proof}
It follows easily from Lemma \ref{invariance_by_X2} that $X$ preserves the quaternionic structure $\bar{Q}$. From Lemma \ref{invariance_by_X2}, Lemma \ref{invariance_by_X4} and the closure of $\hat{\theta}_0$ we do also obtain that $p_{v}[\widehat{X_{P}},Y^{h_{\hat{\theta}^{\prime}}}]=0$, 
where $p_{h}$ and $p_{v}$ denote the projections from $T \hat{M}$ onto the horizontal and vertical 
subbundles, respectively. Using this, for any vector fields $Y$ and $W$ on $\bar{M}$, 
we compute 
\begin{align*}
(L_{X} \bar{\nabla})_{Y} W 
=&\hat{\pi}_{\ast} 
\left( [ \widehat{X_{P}}, \hat{\nabla}^{0}_{Y^{h_{\hat{\theta}^{\prime}}}} W^{\mathcal{\mathcal{\hat{H}}}}]
    -\hat{\nabla}^{0}_{p_{h}[\widehat{X_{P}},Y^{h_{\hat{\theta}^{\prime}}}]} W^{h_{\hat{\theta}^{\prime}}}
    -\hat{\nabla}^{0}_{Y^{h_{\hat{\theta}^{\prime}}}} p_{h}[\widehat{X_{P}} ,W^{h_{\hat{\theta}^{\prime}}}] \right)\\
=& \hat{\pi}_{\ast} \left( (L_{\widehat{X_{P}}} \hat{\nabla}^{0})_{Y^{h_{\hat{\theta}^{\prime}}}} W^{h_{\hat{\theta}^{\prime}}}
    +\hat{\nabla}^{0}_{p_{v}[\widehat{X_{P}},Y^{h_{\hat{\theta}^{\prime}}}]} W^{h_{\hat{\theta}^{\prime}}}
    +\hat{\nabla}^{0}_{Y^{h_{\hat{\theta}^{\prime}}}} p_{v}[\widehat{X_{P}} ,W^{h_{\hat{\theta}^{\prime}}}] \right)
= 0.\qedhere
\end{align*}
\end{proof}

We call the correspondence from  a hypercomplex manifold
$(M,H,Z,f,\Theta)$ to 
a quaternionic manifold  $(\bar{M},\bar{Q},\bar{\nabla},X)$ 
described in 
Theorem \ref{gene_hq_corresp} (and Proposition \ref{affineX} for the additional structure 
$X$) the 
{\cmssl hypercomplex/quaternionic-correspondence}
({\cmssl H/Q-correspondence} for short). 
As we mentioned in Remarks \ref{rem1} 
and \ref{rem2}, the global assumption in Theorem~\ref{gene_hq_corresp} (H/Q\--cor\-res\-pon\-dence) that $\tilde{\pi}$ and $\hat{\pi}$ are submersions is 
always satisfied locally. 
Under stronger assumptions and by considering Swann's twist \cite{Sw}, 
we have the following global result. 
We use the notation $\zeta_{A}$ for the action induced from the group $\langle A \rangle$ 
generated by a vector field $A$ to distinguish $\mathrm{U}(1)$-actions.

\begin{theorem}[H/Q-correspondence, second version]\label{swan_twist}
Let $M$ be a hypercomplex manifold with a hypercomplex structure 
$H=(I_{1},I_{2},I_{3})$, a closed two-form $\Theta$ and a rotating vector field $Z$ 
such that $L_{Z} \Theta=0$. 
Let $f$ be a smooth function on $M$ such that $df = -\iota_{Z} \Theta$ and assume that 
$f_{1}:=f-(1/2)\Theta(Z,I_{1}Z)$ does nowhere vanish. 
Consider a principal $\mathrm{U}(1)$-bundle $\pi:P \to M$ with a connection form $\eta$ 
whose curvature form is 
\[ d \eta
   =\pi^{\ast} \left( \Theta-
   \frac{1}{2} d ((\iota_{Z} \Theta) \circ I_{1}) \right). \] 
If $Z_{1}=Z^{h_{\eta}}+f_{1}X_{P}$ generates 
a free $\mathrm{U}(1)$-action on $P$, then 
the conification $\hat{M}$ of $M$ is $\mathbb{H}^{\ast} \times_{\langle V_{1} \rangle} P$ and 
the quaternionic manifold $\bar{M}$ coincides 
with the twist of $M$ given by the twist data 
$(\Theta-\frac{1}{2} d ((\iota_{Z} \Theta) \circ I_{1}),Z,f_{1})$ as manifolds. 
\end{theorem}

\begin{proof}
By Lemma \ref{deta:lemma}, we see $\iota_{Z} d \eta=-df_{1}$. 
It follows that $L_{Z} d\eta=0$ from 
the assumptions $L_{Z} \Theta=0$ and $L_{Z}I_{1}=0$. 
Therefore we obtain a twist $M^{\prime}:=P / \langle Z_{1} \rangle$ of $M$
with the twist data $(\Theta-\frac{1}{2} d ((\iota_{Z} \Theta) \circ I_{1}),Z,f_{1})$
since $Z_{1}=Z^{h_{\eta}}+f_{1}X_{P}$ generates 
a free $\mathrm{U}(1)$-action.  
%
%
%
Let $\pi^{\prime}:P \to M^{\prime}$ be  
the quotient map by the action of $\langle Z_{1} \rangle$. 
We define 
an action of $\langle V_{1} \rangle (\cong \mathrm{U}(1)) \subset \langle e_{1}^{L} \rangle 
\times \langle Z_{1} \rangle $ 
on $\mathbb{H}^{\ast} \times P$ by 
\[ \zeta_{V_{1}}(u)(z,p)=(\zeta_{e_{1}^{L}}(u)z, \zeta_{Z_{1}}(u^{-1})p) \] 
for $(z,p) \in \mathbb{H}^{\ast} \times P$. 
We see that the conification $\hat{M}$ of $M$ is a fiber bundle 
$(\mathbb{H}^{\ast} \times P)/ \langle V_{1} \rangle$ over $M^{\prime}$, which is associated with 
$\pi^{\prime}:P \to M^{\prime}$ and usually denoted by 
$\mathbb{H}^{\ast} \times_{\langle V_{1} \rangle} P$. 
Moreover the quotient of $\hat{M}$ by $\mathbb{H}^{\ast}$ is $M^{\prime}$, that is, $\bar{M}=M^{\prime}$. 
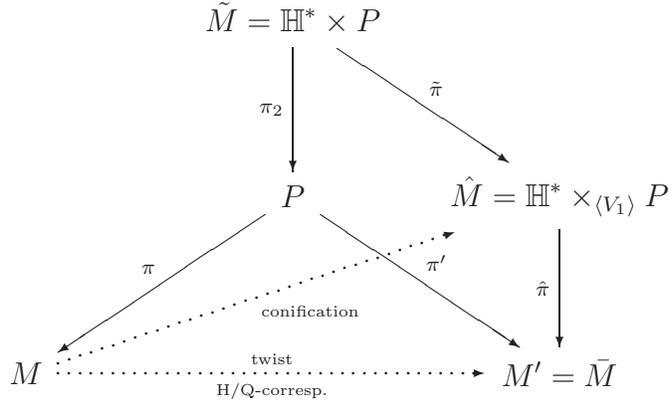
\begin{eqnarray*}
  \begin{diagram}
  \node[3]{\tilde{M}=\mathbb{H}^{\ast} \times P}
  \arrow[1]{s,l}{\pi_{2}}
  \arrow[1]{se,l}{\tilde{\pi}}\\
  \node[3]{P}
  \arrow[1]{sw,l}{\pi} 
  \arrow[1]{se,l}{\pi^{\prime}}
  \node[1]{\hat{M} = \mathbb{H}^{\ast} \times_{\langle V_{1} \rangle} P}
  \arrow[1]{s,l}{\hat{\pi}}\\
  \node[2]{M}
  \arrow[2]{e,tb,..}{\tiny \mbox{twist}}{\tiny \mbox{H/Q-corresp.}}
  \arrow[1]{ene,b,..}{\tiny \mbox{conification}}
  \node[2]{M^{\prime}=\bar{M}}
  \end{diagram}\\
\end{eqnarray*} 
In the above diagram, $\pi_{2}$ is the projection onto the second factor $P$. 
\end{proof}

\begin{remark}\label{sec}
{\rm Note that the bundle $\hat{\pi}: \hat{M}\rightarrow \bar{M}$ is associated to the principal 
$\mathrm{U}(1)$-bundle $P\rightarrow \bar{M}=M'=P/\langle Z_1\rangle$. Therefore sections of $\hat{\pi}$ are in one-to-one correspondence 
with equivariant maps $P\rightarrow \mathbb{H}^{\ast}$.
Let $\lambda:P \to \mathbb{H}^{\ast}$ be such that 
$\lambda(\zeta_{Z_{1}}(u)p)=\zeta_{e_{1}^{L}}(u^{-1})\lambda(p)$
for all $u \in \mathrm{U}(1)$ and $p \in P$ and 
set $F_{\lambda}:=[\lambda, \mathrm{id}]_{\langle V_{1} \rangle} : P \to \hat{M}$. 
If we consider a local section $s:U(\subset \bar{M}=M^{\prime}) \to P$, then 
$s^{\prime}:=F_{\lambda} \circ s : U \to \hat{M}$ 
is a local section of $\hat{\pi}:\hat{M} \to \bar{M}$ 
and the equivariance of $\lambda$ implies that $s'$ is independent of $s$. 
As we observed in the proof of Theorem \ref{gene_hq_corresp}, 
the quaternionic structure $\bar{Q}=\langle \bar{I}_{1}, \bar{I}_{2}, \bar{I}_{3} \rangle$ on $\bar{M}$
is induced from the hypercomplex 
structure on $\hat{M}$ and a local section $s^{\prime}$.   
For $Y \in T_{x} \bar{M}$, we have
\begin{align*}
\bar{I}_{\alpha}(Y)=\hat{\pi}_{\ast}(\hat{I}_{\alpha}Y_{s^{\prime}(x)}^{h_{\hat{\theta}^{\prime}}})
=\hat{\pi}_{\ast}(\hat{I}_{\alpha} s^{\prime}_{\ast} Y), 
\end{align*}
since the decomposition $T \hat{M}={\cal D} \oplus \hat{\cal H}$ is $\hat{I}_{\alpha}$-invariant. 
From 
$s^{\prime}=F_{\lambda} \circ s=[\lambda \circ s, s]_{\langle V_{1} \rangle}
=\tilde{\pi} \circ (\lambda \circ s,s)$, 
it holds that 
\begin{align}
\bar{I}_{\alpha}(Y)&= \hat{\pi}_{\ast}(\hat{I}_{\alpha} s^{\prime}_{\ast} Y) \label{sec_eq1}\\
 &= \hat{\pi}_{\ast} (\hat{I}_{\alpha} (\tilde{\pi}_{\ast} ((\lambda \circ s)_{\ast}(Y)+s_{\ast}Y) ) 
\nonumber \\
 &= \hat{\pi}_{\ast} (\tilde{\pi}_{\ast} ( \tilde{I}_{\alpha} ((\lambda \circ s)_{\ast}(Y)+s_{\ast}Y) ) \nonumber \\
 &= {\pi^{\prime}}_{\ast} ({\pi_{2}}_{\ast} ( \tilde{I}_{\alpha} ((\lambda \circ s)_{\ast}(Y)+s_{\ast}Y) )  \nonumber\\
 &= {\pi^{\prime}}_{\ast} ({\pi_{2}}_{\ast} ( \tilde{I}_{\alpha} s_{\ast}Y) ). \nonumber
\end{align}
Note that $(\lambda \circ s)_{\ast}(Y)+s_{\ast}Y \in T_{(\lambda(s(x)),s(x))}\tilde{M}$. 

Next we consider the decomposition $TP|_{s(U)}=\langle Z_{1} \rangle \oplus s_{\ast}(TU)$. 
Let $p^{\vee}$ be the projection from $TP|_{s(U)}$ onto $s_{\ast}(TU)$. 
Note that $s_{\ast}(T_{x}U)$ is generated by the tangent vectors of the form 
$p^{\vee} (W^{ h_{\eta}}_{s(x)} )=:W^{\vee}$ 
at each point $s(x)$, where $W$ is a tangent vector of $M$ at $\pi(s(x))$
and $\eta$ is the connection form on $P$. 
We define (an almost hypercomplex structure) $I^{\vee}_{\alpha}$ 
on $s(U)$ by 
$I^{\vee}_{\alpha}(W^{\vee})=(I^{\prime}_{\alpha} W)^{\vee}$ 
for each $W^{\vee} \in s_{\ast}(T_{x}U)$, where $I^{\prime}_{\alpha}$ 
is given by (\ref{I_prime}) for $z=\lambda(s(x))$. 
Since
$\tilde{I}_{\alpha}(Z_{1})=\tilde{I}_{\alpha}(e_{1}^{L}) \in T \mathbb{H}^{\ast}$ (by $\tilde{I}_{\alpha}V_{1}=0$), 
we have 
\begin{align}
p^{\vee} (\pi_{2 \ast} (\tilde{I}_{\alpha} (W^{\vee})))
&=p^{\vee} (\pi_{2 \ast} (\tilde{I}_{\alpha} (W^{h_{\eta}}+a Z_{1})))
=p^{\vee} (\tilde{I}_{\alpha} W^{h_{\eta}}) 
=p^{\vee} ((I^{\prime}_{\alpha}W)^{h_\eta}) \label{sec_eq2}\\
&=(I_{\alpha}^{\prime} W)^{\vee} 
=I^{\vee}_{\alpha} (W^{\vee}), \nonumber
\end{align}
where $a \in \mathbb{R}$. 
Then it holds that
\begin{align*}
\bar{I}_{\alpha}(Y)&={\pi^{\prime}}_{\ast} ({\pi_{2}}_{\ast} ( \tilde{I}_{\alpha} s_{\ast}Y) ) 
={\pi^{\prime}}_{\ast} (p^{\vee} ({\pi_{2}}_{\ast} ( \tilde{I}_{\alpha} (s_{\ast}Y)) ) 
={\pi^{\prime}}_{\ast} ( I^{\vee}_{\alpha} (s_{\ast} Y))
\end{align*}
from (\ref{sec_eq1}) and (\ref{sec_eq2}).
Therefore 
$\bar{Q}$ can be identified with $\langle I^{\vee}_{1}, I^{\vee}_{2}, I^{\vee}_{3} \rangle$ on $s(U)$.  
Note that $\langle I^{\vee}_{1}, I^{\vee}_{2}, I^{\vee}_{3} \rangle$ is independent 
of the choice of $\lambda$, and hence 
it is shown again that $\bar{Q}$ is independent of the choice of $\lambda$, which is identified with a section of $\hat{M}$. 
}
\end{remark}

Note that a quaternionic K{\"a}hler metric obtained by the HK/QK-correspondence 
is described directly 
in terms of the objects on $P$ (instead of $\hat{M}$) in \cite{ACDM, MS}.  

\begin{remark}\label{rem2_2}
{\rm 
The conification $\hat{M}$ of $M$ is locally isomorphic to the Swann bundle of $\bar{M}$, which is conical 
as discussed in Example \ref{ex_sw}. 
Note that the Swann bundle is an $\mathbb{H}^{\ast}/\{ \pm 1 \}$-bundle over a quaternionic manifold whereas 
$\bar{M}$ is the quotient of $\hat{M}$ by $\mathbb{H}^{\ast}$ as above. 
Indeed, take an open set $U$ of $\bar{M}$ and local sections $s:U \to \hat{M}$, 
$s^{\prime}:U \to {\cal U}(\bar{M})$, where $\pi^{Sw}:{ \cal U}(\bar{M}) \to \bar{M}$ 
is the Swann bundle of $\bar{M}$. 
For a local trivialization $\Phi:\hat{\pi}^{-1}(U) \to U \times \mathbb{H}^{\ast}$ 
associated to $s$ and given by 
$\Phi(x)=(\hat{\pi}(x),\phi(x))$, we can define a double covering $F:\hat{\pi}^{-1}(U) \to (\pi^{Sw})^{-1}(U) $  
by 
\[ F(x)=\Phi^{\prime -1} ( s^{\prime}(\hat{\pi}(x)) ,  p( \phi(x) )  ). \] 
Here $\Phi^{\prime} : (\pi^{Sw})^{-1}(U) \to U \times \mathbb{H}^{\ast}/\{ \pm 1\}$ 
is a local trivialization associated to $s^{\prime}$ and 
$p:\mathbb{H}^{\ast} \to \mathbb{H}^{\ast}/\{ \pm 1\}$ is the projection. 
See \cite{PPS,BC} for the (twisted) Swann bundle. 
}
\end{remark}

\section{Examples of the H/Q-correspondence}
\label{Examples:sec}
\setcounter{equation}{0}

In this section, we give examples of the H/Q-correspondence.

\begin{example}[HK/QK-correspondence]
{\rm 
Let $(M,g,H=(I_{1},I_{2},I_{3}))$ be a (possibly indefinite) hyper-K{\"a}hler manifold 
with a rotating Killing vector field $Z$ and 
$f$ a nowhere vanishing 
smooth function such that $df=-\iota_{Z} \Theta$, 
where $\Theta$ is the K{\"a}hler form with respect to $g$ and $I_{1}$. 
Set $f_{1}=f-(1/2)g(Z,Z)$ and assume that the functions $g(Z,Z)$ 
and $f_{1}$ are nowhere zero. From these data, we can obtain a (possibly indefinite) quaternionic 
K{\"a}hler manifold 
$(\bar{M},\bar{g})$ \cite{Haydys,ACM,ACDM}. The metric $\bar{g}$ is positive definite under the assumptions specified in \cite[Corollary 2]{ACM} for the signs of the functions $f,f_1$ and for the signature of $g$. Also the sign of the scalar curvature of 
$\bar{M}$ is determined by these choices.}
\end{example}

In the HK/QK-correspondence, the initial data $\Theta$ is a non-degenerate 2-form.
In our more general setting, we may also choose $\Theta=0$, like in the following example. 

\begin{example}[Conical hypercomplex manifold]\label{ex_conical_hyp}
{\rm 
Let $(M,(I_{1},I_{2},I_{3}),V)$ be a conical hypercomplex manifold with the Euler vector field $V$. 
Choose $f_{1}=f=1$, $\Theta=0$, and consider the trivial principal bundle 
$P=M \times \mathrm{U}(1)$ with the connection 
$\eta=dt$, where 
$t$ is the angular coordinate of $\mathrm{U}(1)$ such that $dt(X_{P})=1$ on the fundamental vector field $X_P$. 
We assume that $Z:=I_{1}V$ 
generates a free $\mathrm{U}(1)$-action on $M$
and that 
the periods of $Z$, $X_{P}$ and $e^{L}_{1}$ are the same. 
It holds that $Z$ is rotating  
from Lemma \ref{derivative_D}. 
Then $V_{1}$ generates a free $\mathrm{U}(1)$-action 
on $\tilde{M}=\mathbb{H}^{\ast} \times P= \mathbb{H}^{\ast} \times M \times \mathrm{U}(1)$ 
of the same period. 
Therefore  
\[ \hat{M}(=(\mathbb{H}^{\ast} \times M \times \mathrm{U}(1))/
\langle V_{1} \rangle) \ni [z,p,q]=[zq, \zeta_{Z}(q^{-1})p,1]
\mapsto (zq, \zeta_{Z}(q^{-1})p) \in \mathbb{H}^{\ast} \times M \]
gives 
a diffeomorphism $\hat{M} \cong \mathbb{H}^{\ast} \times M$, 
and hence $\bar{M} \cong M$ as smooth manifolds. 
In fact, we can define a diffeomorphism $\varphi^{\prime}:M \to M^{\prime}(=\bar{M})$ by 
$\varphi^{\prime}(x)=\pi^{\prime}(x,1)$. 
A global section $\bar{M} \to \hat{M}$ 
gives rise to a hypercomplex structure $(\bar{I}_{1},\bar{I}_{2},\bar{I}_{3})$ on $\bar{M}$ 
but the latter does not coincide with
$(I_{1},I_{2},I_{3})$ in general (under the diffeomorphism $\varphi'$). The quaternionic structure 
$\bar{Q}$ on $\bar{M}$, however, coincides with $\langle I_{1}, I_{2}, I_{3} \rangle$.  
Note that $\bar{Q}$
is independent of the section, 
as shown in the proof of Theorem \ref{gene_hq_corresp} 
and Remark~\ref{sec}. 
More explicitly, considering $\lambda_{z}:M \times \mathrm{U}(1) \to \mathbb{H}^{\ast}$ defined 
by $\lambda_{z}(x,u)=z \cdot u^{-1}$ $(z \in \mathbb{H}^{\ast})$ and the section $s:\bar{M} \to P$ 
defined by $s(x)=((\varphi')^{-1}(x),1)$, we see that the section 
$F_{\lambda_{1}} \circ s$ gives the hypercomplex structure 
$(I_1,I_2,I_3)$ and, hence, the 
quaternionic structure 
$\langle I_{1}, I_{2}, I_{3} \rangle$ on $\bar{M}\cong M$.
}
\end{example}

The next example shows that 
our H/Q-correspondence is a proper generalization of the HK/QK-correspondence.

\begin{example}[Hypercomplex Hopf manifold]\label{ex_hopf}
{\rm 
Consider $\mathbb{H}^{n} \cong \mathbb{R}^{4n}$
as a right-vector space over the quaternions 
with the standard hypercomplex structure 
\[ \tilde{H}=(\tilde{I}_{1}=R_i,\tilde{I}_{2}=R_j,\tilde{I}_{3}=\tilde{I}_{1}\tilde{I}_{2}=-R_k)\]  
and the standard flat hyper-K\"ahler metric $\tilde{g}$ 
and set $\tilde{M}=\mathbb{H}^{n}  \backslash \{ 0 \}$. 
Take $A \in \mathrm{Sp}(n) \mathrm{Sp}(1)$ and $\lambda>1$. 
Then $\langle \lambda A \rangle$ is a group of 
homotheties 
which acts freely and properly discontinuously on the 
simply connected manifold $\tilde{M}$. 
The quotient space 
$\tilde{M}/\langle \lambda A \rangle$
inherits a quaternionic structure $Q$ and a quaternionic connection $\nabla$ 
which are invariant under the centralizer $G^{Q}$ of 
$\lambda A$ in $\mathrm{GL}(n,\mathbb{H}) \mathrm{Sp}(1)$. 
In fact, 
the quaternionic structure $\tilde{Q}$ on $\tilde{M}$ is 
${\rm GL}(n,\mathbb H){\rm Sp}(1)$-invariant and 
induces therefore an almost quaternionic structure 
$Q$ on $\tilde{M} / \langle \lambda A \rangle$, 
since $\langle \lambda A \rangle \subset {\rm GL}(n,\mathbb H){\rm Sp}(1)$.  
Moreover, the Levi-Civita connection $\tilde{\nabla}$ 
on $(\tilde{M},\tilde{g})$, 
which coincides with the Obata connection with respect to $\tilde{H}$, 
is invariant under all homotheties of $\tilde{M}$. 
Since $\langle \lambda A \rangle$ acts by homotheties, we see that 
$\tilde{\nabla}$ induces a torsion-free connection $\nabla$ on 
$\tilde{M}/ \langle \lambda A \rangle$, which preserves $Q$. 
This means that $Q$ is a quaternionic
structure on $\tilde{M} / \langle \lambda A \rangle$. 
In particular, if $A \in \mathrm{Sp}(n)$, then the quotient 
$\tilde{M}/\langle  \lambda A \rangle$
inherits an induced 
hypercomplex structure $H=(I_{1},I_{2},I_{3})$ from $\tilde{H}$, 
which is invariant under the centralizer $G^{H}$ of  
$\lambda A$ in $\mathrm{GL}(n,\mathbb{H})$, 
since $\langle  \lambda A \rangle$ preserves $\tilde{H}$. 
We say that $(\tilde{M}/\langle \lambda A \rangle,Q)$ 
(resp. $(\tilde{M}/\langle \lambda A \rangle,H)$) 
is a quaternionic (resp. hypercomplex) Hopf manifold. See \cite{OP,CH}.   

We start with  a hypercomplex Hopf manifold 
$M:=\tilde{M}/\langle  \lambda A  \rangle$, where 
$A \in \mathrm{Sp}(n)$. 
Take $q \in \mathrm{Sp}(1)$ such that $q \neq \pm 1$. 
The centralizer of $q$ in $\mathrm{Sp}(1)$ is isomorphic to $\mathrm{U}(1)$, which is denoted by $\mathrm{U}_{q}(1)$.
We consider a $\mathrm{U}(1)$-action : $z \mapsto z e^{-it}$
on $\tilde{M}$ defined by the right multiplication of 
$\mathrm{U}(1) \cong \mathrm{U}_{q}(1) \subset \mathrm{Sp}(n) \mathrm{U}_{q}(1) 
\subset \mathrm{Sp}(n) \mathrm{Sp}(1)$. 
This action induces one on $M$ and the corresponding vector field $Z$ is rotating. 
Therefore we can apply the same procedure as in Example \ref{ex_conical_hyp} 
under the setting $P=M \times \mathrm{U}_{q}(1)$ (resp. $\tilde{P}=\tilde{M} \times 
\mathrm{U}_{q}(1)$) and $\Theta=0$, and we
have the quaternionic manifold $\bar{M}(=M^{\prime})$ 
(resp. $\bar{\tilde{M}}(=\tilde{M}^{\prime})$)
by the H/Q-correspondence. 
In the following, the quotient map of an action by a group $G$ is denoted by $\pi_{G}$ 
and the objects associated with $\tilde{M}$ are denoted by the corresponding letters for $M$ with 
$\tilde{\,\,\,}$, for example, the projection of the twist from $\tilde{M} \times \mathrm{U}_{q}(1)$ 
is denoted as $\tilde{\pi}^{\prime}$, where we use the notation of Theorem \ref{swan_twist}. 
Let $R_{q}$ be the right multiplication by $q$. 
\begin{eqnarray*}
  \begin{diagram}
  \node[1]{\tilde{M}^{\prime}=\bar{\tilde{M}}=\tilde{M}}
  \arrow[1]{e,b}{/ \langle  \lambda A R_{q}  \rangle}
  \node[1]{M^{\prime}=\bar{M}}\\
  \node[1]{\tilde{P}=\tilde{M} \times \mathrm{U}_{q}(1)}
  \arrow[1]{n,l}{\tilde{\pi}^{\prime}}
  \arrow[1]{e,b}{/ \langle  \lambda A  \rangle}
  \arrow[1]{s,l}{\tilde{\pi}}
  \node[1]{P=M \times \mathrm{U}_{q}(1)}
   \arrow[1]{n,l}{\pi^{\prime}}
  \arrow[1]{s,l}{\pi}\\
  \node[1]{\tilde{M}}
  \arrow[1]{e,b}{/ \langle  \lambda A  \rangle} 
  \node[1]{M}
\end{diagram}\\
\end{eqnarray*}

\noindent
Since $\pi^{\prime} \circ \pi_{\langle \lambda A \rangle}
=\pi_{\langle \lambda A R_{q} \rangle} \circ \tilde{\pi}^{\prime}$ and 
$\tilde{M}^{\prime}=\tilde{M}$ is a manifold with an invariant quaternionic structure under the action 
of $\langle \lambda A R_{q} \rangle$ (Example \ref{ex_conical_hyp} and Proposition \ref{affineX}), 
we have
\[  \bar{M}=M^{\prime}=\tilde{M}/\langle \lambda A R_{q} \rangle. \]
Therefore it holds that 
\[ M=\tilde{M}/\langle  \lambda A  \rangle \; \overset{\text{\tiny H/Q\;\;}}{\longmapsto}\; \bar{M}=\tilde{M}/\langle \lambda A R_{q} \rangle.
\] 

In particular, we can choose $A=E_{n} \in \mathrm{Sp}(n)$. 
Then the centralizer $G^{H}$ of 
$\lambda=\lambda E_{n}$ is $\mathbb{R}^{>0} \times \mathrm{SL}(n,\mathbb{H})$. We take the
subgroup $\mathbb{R}^{>0} \times \mathrm{Sp}(n)$ of $G^{H}$, which acts transitively on $M$. Then 
\[ M=(\mathbb{R}^{>0}/ \langle \lambda \rangle) \times \frac{\mathrm{Sp}(n)}{\mathrm{Sp}(n-1)}. \]
On the other hand, considering the subgroup $\mathbb{R}^{>0} \times \mathrm{Sp}(n) \mathrm{U}_{q}(1)$ 
of the centralizer $G^{Q}$ of $\lambda R_{q}$, we see that
\[ (\mathbb{R}^{>0}/ \langle \lambda \rangle) \times \frac{\mathrm{Sp}(n)}{\mathrm{Sp}(n-1)} \;\overset{\text{\tiny H/Q\;\;}}{\longmapsto}\;
(\mathbb{R}^{>0}/ \langle \lambda \rangle) 
\times \frac{\mathrm{Sp}(n) \mathrm{U}(1)}{\mathrm{Sp}(n-1) \triangle_{\mathrm{U}(1)}}, \]
where $\triangle_{\mathrm{U}(1)}$ is a diagonally embedded subgroup of 
$\mathrm{Sp}(n) \mathrm{U}(1) \subset \mathrm{Sp}(n) \mathrm{Sp}(1)$ which is 
isomorphic to $\mathrm{U}(1)$. 
Considering the case of $n=2$, we have an invariant quaternionic structure 
on the homogeneous space 
\[ \bar{M}=\mathbb{R}^{>0}/ \langle \lambda \rangle 
\times \frac{\mathrm{Sp}(2) \mathrm{U}(1)}{\mathrm{Sp}(1) \triangle_{\mathrm{U}(1)}}
=\frac{T^{2} \cdot \mathrm{Sp}(2)}{\mathrm{U}(2)} \]
by the  H/Q-correspondence. 
Note that $T^{2} \times \mathrm{Sp}(2)$ carries a hypercomplex structure and 
$(T^{2} \times \mathrm{Sp}(2)) / \mathrm{U}(2)$ 
is a homogeneous quaternionic manifold considered in \cite{J2}.

Since $M$ 
is diffeomorphic to $S^{1} \times S^{4n-1}$, $M$ can not admit any 
hyper-K\"ahler structure. Therefore 
the HK/QK-correspondence can not be applied to 
the hypercomplex Hopf manifold $M$. 
The H/Q-correspondence is thus a proper generalization of the HK/QK one. 
}
\end{example}

In the following example, the closed form $\Theta$ is non-zero and degenerate.  

\begin{example}[Lie group with left-invariant hypercomplex structure]\label{ex_su_3}
{\rm
Consider $G=\mathrm{SU}(3)$. 
The Lie algebra $\mathfrak{g}$ of $G$ is decomposed as
$\mathfrak{g}=\mathfrak{g}_{0} + \mathfrak{g}_{1}$, where 
$\mathfrak{g}_{0} =  \mathfrak{s}(\mathfrak{u}(1) \oplus \mathfrak{u}(2)) 
\cong \mathfrak{u}(1) \oplus \mathfrak{su}(2)
\cong \mathbb{H}$ and $\mathfrak{g}_{1}$ is the unique complementary 
$\mathfrak{g}_{0}$-module 
with the action of $\mathbb{H}$ obtained from the 
adjoint action of $\mathfrak{g}_{0}$ \cite{J2}. 
Denote by $V \in \mathfrak{g}_{0}$ the vector which corresponds to $1 \in \mathbb{H}$. 
We use the same letters for left-invariant vector fields and corresponding 
elements of $\mathfrak{g}$ in this example. 
Three complex structures $I_{1}, I_{2}, I_{3}$ on $\mathfrak{g}$ 
can be defined as follows. 
They preserve the decomposition $\mathfrak{g}=\mathfrak{g}_{0} + \mathfrak{g}_{1}$ and act on $\mathfrak{g}_{0}=\mathbb{H}$ by the standard hypercomplex structure $(R_i,R_j,R_iR_j=-R_{k})$. On $\mathfrak{g}_{1}$ they are defined by
\begin{align}\label{ex_su3_hypcpx}
{I_\alpha}|_{\mathfrak{g}_{1}}= -\mathrm{ad}_{I_\alpha V}|_{\mathfrak{g}_{1}},\quad \alpha =1,2,3.
\end{align}
These structures extend to a left-invariant hypercomplex structure on $G$ \cite{J2}, which we denote again by 
$(I_{1}, I_{2}, I_{3})$.

Let $G_{0}\cong (\mathrm{U}(1) \times \mathrm{SU}(2))/\{ \pm \mathrm{1}\}\cong \mathrm{U}(2)$ be the subgroup of $G$ corresponding to $\mathfrak{g}_{0}$. 
Note that $G_0\subset G$ is a hypercomplex submanifold and therefore totally geodesic 
with respect to the Obata connection $\nabla^{G}$ of $G$ \cite{PPS}. 
The Obata connection $\nabla^{G_{0}}$ of $G_{0}$ is given by  
$\nabla^{G_{0}}_{X}Y=XY$ for $X$, $Y \in \mathfrak{g}_{0}=\mathbb{H}$, 
where $XY$ denotes the product of the quaternions $X$ and $Y$. 
Indeed, $\nabla^{G_{0}}$ is torsion-free and $I_{1}$, $I_{2}$, $I_{3}$ are parallel with respect to 
$\nabla^{G_{0}}$. Then it holds 
$\nabla^{G}_{X} V = \nabla^{G_{0}}_{X} V =X$ for $X \in \mathfrak{g}_{0}$. 
For $X \in \mathfrak{g}_{1}$, by $(\ref{ex_su3_hypcpx})$ and the explicit expression of
the Obata connection (see \cite{AM}), 
we also find that $\nabla^{G}_{X} V =X$. Hence the hypercomplex manifold 
$(G,(I_1,I_2,I_3))$
is conical with the Euler vector field $V$ (see also \cite{So}). 

Consider the right-action of  $\mathrm{U}(2)$ on $\mathrm{SU}(3)$ given by 
\[ AB:=A \left( \begin{array}{cc}
       B  &  0  \\
       0  &  \det(B)^{-1}  \\
       \end{array} \right) \]
for $A \in \mathrm{SU}(3)$ and $B \in \mathrm{U}(2)$. 
Let $l:\mathrm{SU}(3) \to \mathrm{SU}(3)/\mathrm{U}(2) \cong \mathbb{C}P^{2}$ 
be the projection and $k:S^{5} \to \mathbb{C}P^{2}$ 
the Hopf fibration. The pullback bundle $P:=l^{\#}S^{5}$ 
of $k:S^{5} \to \mathbb{C}P^{2}$
by $l$ is a $\mathrm{U}(1)$-bundle over $\mathrm{SU}(3)$. 
The usual identification between 
the Stiefel manifold $V_{2}(\mathbb{C}^{3})$ and $\mathrm{SU}(3)$ is given by  
\[ V_{2}(\mathbb{C}^{3}) \ni (a_{1},a_{2}) \leftrightarrow 
A=(a_{1},a_{2}, \bar{a}_{1} \times \bar{a}_{2}) \in \mathrm{SU}(3). \] 
We can write 
\begin{align*}
P &=\{ (A , u) \in \mathrm{SU}(3) \times S^{5} 
\mid l(A)=k(u) \} \\
   &=\{ (A , u) \in \mathrm{SU}(3) \times S^{5} 
\mid \langle c_{3}(A) \rangle 
= \langle u \rangle \in \mathbb{C}P^{2} \} \\
&=\{ (A , \alpha c_{3}(A)) 
   \in \mathrm{SU}(3) \times S^{5} 
\mid \alpha \in \mathrm{U}(1) \} \\
& \cong \mathrm{SU}(3) \times \mathrm{U}(1), 
\end{align*}
where 
$c_{3}(A)$ denotes the third column of $A$. 
This shows that $P$ is a trivial bundle.
Let $l_{\#}:P \to S^{5}$ be the bundle map given by 
$l_{\#}(A,\alpha)=\alpha(\bar{a}_{1} \times \bar{a}_{2})=\alpha \, c_{3}(A)$. 
Consider the pullback connection $l_{\#}^{\ast} \eta$ 
on $P$ from the standard one $\eta$ of $k$ and take 
$\Theta=l^{\ast} \omega$, where $\omega$ is the K{\"a}hler form on 
$\mathbb{C}P^{2}$.
Set $Z:=I_{1} V$. We see that $Z$ generates a $\mathrm{U}(1)$-action on 
$\mathrm{SU}(3)$ 
and is rotating   
by Lemma~\ref{derivative_D}. 
Since 
\[ \langle Z \rangle  \subset \mathrm{SU}(2) \subset 
\mathrm{U}(2), \]
$Z$ is tangent to the fiber of $l$. 
Hence, we have $\iota_{Z} \Theta =0$, $L_{Z} \Theta=0$, and 
also have $d \Theta=0$ by $d \omega=0$. 
So we can choose $f=f_{1}=1$ (see Section~\ref{conif:sec} for the notation) and then see that 
$Z_{1}$ generates a free $\mathrm{U}(1)$-action on $P$
given by 
\[ \zeta_{Z_{1}}(u)(A,  \alpha )
=(\zeta_{Z}(u)(A), u \alpha),\quad u\in \mathrm{U}(1). 
\] 
To see this, it is sufficient to check that $Z$ is horizontal with respect to the 
pull back connection. 
The vector field $Z$ is lifted to $\mathrm{SU}(3) \times \mathrm{U}(1)$ as 
$Z_{(A,\alpha)}=(Z_{A},0) \in T\mathrm{SU}(3) \times T \mathrm{U}(1)$ 
for $A \in \mathrm{SU}(3)$ and $\alpha \in \mathrm{U}(1)$ with the same letter $Z$. 
From $Z \in \mathfrak{su}(2)$, it holds that 
\begin{align*}
 l_{\# \ast} Z_{(A,\alpha)} 
&= \left. \frac{d}{dt} l_{\#} \left( 
\zeta_{Z}(e^{it}) (A ,  \alpha) \right)
\right|_{t=0} \\
&= \left. \frac{d}{dt} l_{\#}( ( \zeta_{Z}(e^{it})(A) , \alpha) )
\right|_{t=0} \\
&= \left. \frac{d}{dt} \alpha \, c_{3}(\zeta_{Z}(e^{it})(A))
\right|_{t=0} \\
&= \left. \frac{d}{dt} \alpha \, c_{3}(A)
\right|_{t=0}=0.
\end{align*} 
In particular, $(l_{\#}^{\ast} \eta)(Z)=0$, that is,  $Z$ is 
horizontal with respect to the pullback connection. 
So we see that $Z_{1}=Z+X_{P}$.   
Therefore, by applying the H/Q-correspondence to $G=\mathrm{SU}(3)$, 
we have a quaternionic manifold 
\[ 
\bar{G}=P/ \langle Z_{1} \rangle 
= \left( \mathrm{SU}(3) \times \mathrm{U}(1)\right) / \mathrm{U}(1) 
\cong \mathrm{SU}(3). 
 \]
The identification is given by  
\[ \left( \mathrm{SU}(3) \times \mathrm{U}(1)\right) / \mathrm{U}(1) 
\ni [(A,\alpha)]_{\langle Z_{1} \rangle}=[ (\zeta_{Z}(\alpha^{-1})A, 1) ]_{\langle Z_{1} \rangle}
\cong  \zeta_{Z}(\alpha^{-1})A \in \mathrm{SU}(3). \]   
Note that there exists no Riemannian metric $g$ on $G$ such that 
$g$ is hyper-K{\"a}hlerian with respect to $(I_{1},I_{2},I_{3})$ since 
$G$ is compact. The situation is summarized in the following diagram.
\begin{eqnarray*}
  \begin{diagram}
  \node[2]{P=l^{\#}S^{5} \cong \mathrm{SU}(3) \times \mathrm{U}(1)}
  \arrow[1]{sw}
  \arrow[2]{s,l,3}{l_{\#}} 
  \arrow[1]{se}\\
  \node[1]{\mathrm{SU}(3)=G}
  \arrow[2]{e,b,3,..}{\tiny \mbox{H/Q-corresp.}}
  \arrow[1]{s,l}{l}
  \node[2]{\bar{G}}\\
  \node[1]{\mathrm{SU}(3) / \mathrm{U}(2) \cong \mathbb{C}P^{2}}
  \node[1]{S^{5}}
  \arrow[1]{w,b}{k} 
\end{diagram}\\
\end{eqnarray*}

{\rm
Note also that $\mathrm{SU}(3) \times \mathrm{U}(1)$
is a three-fold covering of $\mathrm{U}(3)$ : $(A,\alpha) \mapsto \alpha A$. 
The kernel is the cyclic group
$\{ (\zeta \, \mathrm{1},\zeta^{-1})\mid \zeta^3 =1\}$. The principal bundle $P\rightarrow \mathrm{SU}(3)$
induces a principal bundle $\mathrm{U}(3) =P/\mathbb{Z}_3 \rightarrow \mathrm{PSU}(3)=\mathrm{SU}(3)/\mathbb{Z}_3$. 
The actions generated by $Z_{1}$ and $Z$ commutes with that of $\mathbb{Z}_{3}$. 
The vector field $Z$ (resp. $Z_{1}$) on $\mathrm{SU}(3)$ 
(resp.\ $\mathrm{SU}(3) \times \mathrm{U}(1)$)
induces one on $\mathrm{PSU}(3)$ (resp.\ $\mathrm{U}(3)$), 
which is denoted by the same letter $Z$ (resp.\ $Z_{1}$) .  
We obtain the following diagram. 
\begin{eqnarray*}
  \begin{diagram}
  \node[1]{P / \langle Z_{1} \rangle = \bar{G} \cong \mathrm{SU}(3)}
  \arrow[1]{e,b}{/ \mathbb{Z}_{3}}
  \node[1]{\mathrm{U}(3)/\langle Z_{1} \rangle = \bar{G}_{1} \cong \mathrm{SU}(3)/\mathbb{Z}^{3}}\\
  \node[1]{P=\mathrm{SU}(3) \times \mathrm{U}(1)}
  \arrow[1]{n,l}{/ \langle Z_{1} \rangle }
  \arrow[1]{e,b}{/  \mathbb{Z}_{3}}
  \arrow[1]{s,l}{/ \mathrm{U}(1)} 
  \node[1]{P/\mathbb{Z}_{3}=\mathrm{U}(3)}
   \arrow[1]{n,l}{ / \langle Z_{1} \rangle   }
  \arrow[1]{s,l}{/ \mathrm{U}(1)}\\
  \node[1]{G=\mathrm{SU}(3)}
  \arrow[1]{e,b}{/\mathbb{Z}_{3} } 
  \node[1]{G_{1}=\mathrm{PSU}(3)}
\end{diagram}\\
\end{eqnarray*}
We can apply the H/Q-correspondence to 
the Lie group $G_1=\mathrm{PSU}(3)$ with 
the induced left-invariant hypercomplex structure and see that 
its resulting space is $\mathrm{SU}(3)/ \mathbb{Z}^{3}$. 
In fact, since the action of $\langle Z_{1} \rangle$ on $\mathrm{U}(3)$ is given by 
$\zeta_{Z_{1}}(u)(\alpha A)=(u \alpha) (\zeta_{Z}(u)(A))$ and 
its orbit $\{ (u \alpha) (\zeta_{Z}(u)(A)) \mid u \in \langle Z_{1} \rangle \}$ of 
$\alpha A \in \mathrm{U}(3)$
intersects $\mathrm{SU}(3)$ at exactly three points, 
then the resulting space $\mathrm{U}(3)/\langle Z_{1} \rangle$ is $\mathrm{SU}(3)/ \mathbb{Z}^{3}$. 
Consequently, we have $\bar{G}_1 \cong G_1$ again.
}

Next we compare the quaternionic structures on the resulting space(s) 
derived from the pullback connection $\eta_{1}$, 
which is not flat, and the trivial connection $\eta_{0}$ 
as in Example~\ref{ex_conical_hyp}. Recall the notation in Remark \ref{sec}. 
We claim that the two quaternionic structures are different. 
We label the objects obtained from $\eta_{i}$ by the symbol $\eta_i$ or just 
by the letter $i$ ($i=0$, $1$), when no confusion is possible. 
Since $Z^{h_{\eta_{0}}}=Z^{h_{\eta_{1}}}$, $\iota_{Z} \Theta_{0}=\iota_{Z} \, 0 =0$ 
and $\iota_{Z} \Theta_{1}=0$, 
the vector field $Z_{1}$ on $P$ is $Z_{1}=Z+X_{P}$ for both connections $\eta_{0}$ and $\eta_{1}$.  
Then the resulting spaces $\bar{G}^{0}$ and $\bar{G}^{1}$ coincide and we simply write $\bar{G}$ for both. 
%
Let $a$ be the $1$-form on $\bar{G}$ such that $\eta_{1}-\eta_{0}=\pi^{\ast} a$. 
Consider a local section $s:\bar{G} \to P$.  
Since $W^{h_{\eta_{1}}}-W^{h_{\eta_{0}}}=-a(W)X_{P}$ for a tangent vector $W$ at 
$\pi(s(x)) \in \bar{G}$ 
(we omit the reference points of tangent vectors), 
we have
\[ W^{\vee 1}-W^{\vee 0}=-a(W) \mathfrak{X} \]
where $\mathfrak{X}=p^{\vee} (X_{P})$ and we recall that $W^{\vee i} = p^{\vee}(W^{h_{\eta_{i}}})$.  
Therefore we see that 
\begin{align*}
{I}^{\vee 1}_{\alpha} (W^{\vee 1})
&=(I^{\prime}_{\alpha}W)^{\vee 1} \\
&=(I^{\prime}_{\alpha}W)^{\vee 0} -a(I^{\prime}_{\alpha}W) \mathfrak{X}\\
&=I^{\vee 0}_{\alpha} (W^{\vee 0}) -a(I^{\prime}_{\alpha}W) \mathfrak{X}\\ 
&=I^{\vee 0}_{\alpha} (W^{\vee 1})+a(W)I^{\vee 0}_{\alpha} 
\mathfrak{X} -a(I^{\prime}_{\alpha}W)\mathfrak{X}. 
\end{align*}
On the other hand, since 
$W^{\vee 1}=W^{h_{\eta_{1}}}+cZ_{1}=W^{h_{\eta_{1}}}+c(Z^{h_{\eta_{1}}}+X_{P})$, 
we have $\eta_{1}(W^{\vee 1})=c$ and $\pi_{\ast}(W^{\vee 1})=W+cZ$. 
It holds  that 
\[ (\pi^{\ast}a)(W^{\vee 1})=a(W)+ca(Z)=a(W)+a(Z) \eta_{1}(W^{\vee 1}). \]
Hence we have 
\[ {I}^{\vee 1}_{\alpha}=I^{\vee 0}_{\alpha} 
+(\pi^{\ast} a - a(Z) \eta_{1} ) \otimes (I^{\vee 0}_{\alpha} \mathfrak{X})
- ((\pi^{\ast} a - a(Z) \eta_{1} )
\circ I^{\vee1 }_{\alpha}) \otimes \mathfrak{X}. \]
Set $\rho:=\pi^{\ast} a - a(Z) \eta_{1} $ and 
$A:=\rho \otimes  (I^{\vee 0}_{\alpha} \mathfrak{X} ) -
(\rho \circ  I^{\vee1 }_{\alpha}) \otimes \mathfrak{X}$.
If $Q^{\vee 0}(:=\langle I^{\vee 0}_{1}, I^{\vee 0}_{2}, I^{\vee 0}_{3} \rangle)
= Q^{\vee 1}(:=\langle I^{\vee 1}_{1} , I^{\vee 1}_{2}, I^{\vee 1}_{3} \rangle)$, then 
$A^{2}=-|A|^{2} \mathrm{id}$, where $| \, \cdot \, |$ is the norm induced from the metric on $Q^{\vee 0}$ such that 
$I^{\vee 0}_{1}, I^{\vee 0}_{2}, I^{\vee 0}_{3}$ are orthonormal. As the rank of $A$ is at most $2$, 
this is only possible if $A=0$. 
This implies  $\rho = \pi^{\ast} a - a(Z) \eta_{1} =0$, which is equivalent to $a=0$. 
By Remark \ref{sec}, the quaternionic structure $\bar{Q}^{i}$ can be identified with 
$Q^{\vee i}$ $(i=0,1)$. 
Then we see that $\bar{Q}^{0} \neq  \bar{Q}^{1}$ since $\eta_{0} \neq \eta_{1}$. 
This proves the claim. 
}
\end{example}

\section{The tangent bundle of a special complex manifold and a generalization of the rigid c-map}
\setcounter{equation}{0}

In this section, we consider a generalization of the  rigid c-map \cite{CFG,Freed, ACD}. 
The generalization associates 
a hypercomplex manifold $M$, the Obata connection of which is Ricci-flat, 
with a special complex manifold.  
In the case of a 
{\it conical} special complex manifold, we shall show that the hypercomplex manifold carries a canonical 
rotating 
vector field $Z^M$ (Lemma \ref{rotation_eq}), such that we can apply our H/Q correspondence. 
Consequently, we shall construct a quaternionic manifold from 
a conical special complex manifold as the generalized supergravity c-map (Theorem \ref{gene_sup_cmap}). 
We start with defining a class of manifolds generalizing
conical special K\"ahler manifolds \cite{ACD,MS}.

\begin{definition}\label{scm:def}
A {\cmssl special complex manifold} $(N,J,\nabla)$ is a complex manifold $(N,J)$ endowed with a 
torsion-free flat connection $\nabla$ such that the $(1,1)$-tensor field $\nabla J$ is symmetric.
A {\cmssl conical special complex manifold}  $(N,J,\nabla, \xi)$ is a special complex manifold $(N,J,\nabla)$
endowed with a vector field $\xi$ such that 
\begin{itemize}
\item $\nabla \xi=\mathrm{id}$ and 
\item $L_{\xi} J=0$ or, equivalently, $\nabla_\xi J =0$.
\end{itemize}
\end{definition}

The connection $\nabla$ is called the {\cmssl special connection}. To see that $L_{\xi} J=0$ is equivalent to $\nabla_\xi J=0$ it suffices to write 
$L_\xi = \nabla_\xi -\nabla \xi = \nabla_\xi - \mathrm{id}$, using that $\nabla$ is torsion-free and $\nabla\xi = \mathrm{id}$.
We also note that the integrability of $J$ follows from the symmetry of $\nabla J$ since $\nabla$ is torsion-free. 
We set $A:=\nabla J.$

\begin{lemma}\label{sc:lem}
For every conical special complex manifold, we have $L_{J \xi} J=A_{J\xi }=0$.
\end{lemma}

\begin{proof}
Based on the symmetry of $\nabla J$, we compute 
\[ A_{J\xi} = A(J\xi) = -J(A\xi) = -J A_\xi =0.\]
Using this and the properties listed in Definition \ref{scm:def}, 
we then obtain 
\begin{align*}
(L_{J\xi} J)X
=&- A_{JX} \xi +J A_{X}\xi=0
\end{align*}
for all $X \in \Gamma(TN)$. 
Note that in the last step we have used the symmetry of $A=\nabla J$. 
\end{proof}

Next we consider the tangent bundle $TN=:M$ of a special complex manifold $(N,J,\nabla)$. We can define the $\nabla$-horizontal lift $X^{h_{\nabla}}$
and the vertical lift $X^{v}$
of $X \in \Gamma(TN)$. See \cite{Bl} for example. 
The $C^\infty (M)$-module $\Gamma (TM)$ is generated by vector fields of the form $X^{h_{\nabla}}+Y^{v}$, 
where $X$, $Y \in \Gamma(TN)$.
On $M$, we define a triple of $(1,1)$-tensors 
$(I_{1},I_{2},I_{3})$ by 
\begin{align}
I_{1}(X^{h_{\nabla}}+Y^{v}) &=(JX)^{h_{\nabla}}-(JY)^{v}, \label{I_1}\\
I_{2}(X^{h_{\nabla}}+Y^{v}) &= Y^{h_{\nabla}}-X^{v}, \label{I_2}\\
I_{3}(X^{h_{\nabla}}+Y^{v}) &= (J Y)^{h_{\nabla}}+(JX)^{v} \label{I_3}
\end{align}
for $X^{h_{\nabla}}+Y^{v} \in TM$. 
Note that $(I_{1},I_{2},I_{3})$ is an almost hypercomplex structure. 
In fact, it is easy to see $I_{\alpha}^{2}=-\mathrm{id}$ and 
\begin{align*}
(I_{1} \circ I_{2})(X^{h_{\nabla}}+Y^{v}) &=I_{1}(Y^{h_{\nabla}}-X^{v})=(JY)^{h_{\nabla}}+(JX)^{v}=I_{3}(X^{h_{\nabla}}+Y^{v}),  \\
(I_{2} \circ I_{1})(X^{h_{\nabla}}+Y^{v}) &=I_{2}( (JX)^{h_{\nabla}}-(JY)^{v})=-(JY)^{h_{\nabla}}-(JX)^{v}=
-I_{3}(X^{h_{\nabla}}+Y^{v})
\end{align*}
for $X^{h_{\nabla}}+Y^{v} \in TM$.
Note
that it holds 
\begin{align}\label{bracket}
[X^{h_{\nabla}},Y^{h_{\nabla}}] =[X,Y]^{h_{\nabla}}, \,
[X^{h_{\nabla}},Y^{v}] =(\nabla_{X} Y)^{v}, \, 
[X^{v},Y^{v}] =0
\end{align}
for $X$, $Y \in \Gamma(TN)$.

\begin{lemma}\label{ingrable_special_hyp}
For every special complex manifold $(N,J,\nabla)$, the canonical almost hypercomplex structure 
$(I_{1},I_{2},I_{3})$ on $M=TN$ is integrable, that is, 
$(M,(I_{1},I_{2},I_{3}))$ is a hypercomplex manifold. 
\end{lemma}

\begin{proof}
Thanks to (\ref{bracket}), the Nijenhuis tensors of $I_{1}$ and $I_{2}$ can be 
easily calculated and we find the following. Using that $J$ is integrable, $\nabla$ is flat and $\nabla J$ is symmetric, we see that $I_{1}$ is integrable. 
Because $\nabla$ is flat and torsion-free, $I_{2}$ is integrable. 
The integrability of $I_{3}$ follows from that of 
$I_{1}$ and $I_{2}$ \cite[Theorem 3.2]{AM}.
\end{proof}

We define a connection $\nabla^{\prime}$ by 
\[ \nabla^{\prime}:=\nabla -\frac{1}{2} J(\nabla J)=\nabla -\frac{1}{2} J A.
\]
Then we see that $\nabla^{\prime}J=0$ and $\nabla^{\prime}$ is torsion-free for every special
complex manifold.  
Moreover, when the special complex manifold is conical, it holds that $\nabla^{\prime} \xi=\nabla \xi=\mathrm{id}$. 

\begin{lemma}\label{curvature_p}
For every special complex manifold $(N,J,\nabla)$, we have
\begin{align*}
R^{\nabla^{\prime}}_{X,Y}=-\frac{1}{4}[A_{X},A_{Y}] 
\end{align*}
for $X$, $Y \in TN$.
\end{lemma}

\begin{proof}
Set $S=-(1/2) J(\nabla J)$. 
Since $\nabla$ is flat, we see that
 the curvature $R^{\nabla^{\prime}}$ of  $\nabla^{\prime}$ is given by
\[ R^{\nabla^{\prime}}_{X,Y}=(\nabla_{X}S)_{Y}-(\nabla_{Y}S)_{X}
+[S_{X},S_{Y}] \]
for $X$, $Y \in TN$. By 
\begin{align*}
(\nabla_{X}S)_{Y}-(\nabla_{Y}S)_{X}
&=-\frac{1}{2} [A_{X},A_{Y}]
   -\frac{1}{2} J(R^{\nabla}_{X,Y}J), \\
[S_{X},S_{Y}]
&=\frac{1}{4} [A_{X},A_{Y}], 
\end{align*}
we have the conclusion. 
\end{proof}

Hence a special complex manifold admits the complex connection 
$\nabla^{\prime}$ such that $R^{\nabla^{\prime}}$ is of type $(1,1)$. 
In fact, it follows from $A_{JX}=-J A_{X}$ for all $X \in TN$. 
The following theorem is a generalization of the rigid c-map in the absence of a metric.

\begin{theorem}[Generalized rigid c-map]\label{ricci_flat} 
The tangent bundle of any special complex manifold $(N,J,\nabla)$ 
carries a canonical hypercomplex structure, defined by (\ref{I_1})-(\ref{I_3}), and 
the Obata connection of the hypercomplex manifold $(M=TN, (I_{1},I_{2},I_{3}))$ 
is Ricci flat.
\end{theorem}

\begin{proof} The integrability of the canonical almost hypercomplex structure 
defined by (\ref{I_1})-(\ref{I_3}) was proven in Lemma~\ref{ingrable_special_hyp}.
Let $\tilde{\nabla}^{0}$ be its  Obata connection.  Using the explicit expression of 
the Obata connection, we have 
\begin{align*}
&{\tilde{\nabla}^{0}}_{X^{h_{\nabla}}}Y^{h_{\nabla}}
=(\nabla^{\prime} _{X}Y )^{h_{\nabla}}, \,\,\,
{\tilde{\nabla}^{0}}_{U^{v}} X^{h_{\nabla}} 
= -\frac{1}{2}(JA_{X} U)^{v} 
= -\frac{1}{2}(JA_{U}X)^{v}\\
&{\tilde{\nabla}^{0}}_{X^{h_{\nabla}}} U^{v} =(\nabla^{\prime}_{X} U)^{v}, \,\,\, 
{\tilde{\nabla}^{0}}_{U^{v}} V^{v} 
= \frac{1}{2}(J A_{V} U)^{h_{\nabla}} 
= \frac{1}{2}(J A_{U} V)^{h_{\nabla}}
\end{align*}
for $X$, $Y$, $U$, $V \in \Gamma(TN)$. 
It can be also checked directly, using by (\ref{I_1})-(\ref{bracket}), that the above formulas for 
$\tilde{\nabla}^{0}$ 
on horizontal and vertical lifts extend uniquely to a torsion-free connection $\tilde{\nabla}^{0}$ for which $I_{1}$, $I_{2}$, $I_{3}$ are parallel. 
We see that the bundle projection from $(TN, \tilde{\nabla}^{0})$ onto $(N,\nabla^{\prime})$ is an 
affine submersion \cite{AH}. 
Again, a straightforward calculation (or 
the fundamental equations of an affine submersion) gives 
\begin{align*}
R^{{\tilde{\nabla}^{0}}}_{U^{v},V^{v}}W^{v}
=&{-\frac{1}{4}(A_{U}A_{V}W)^{v}+\frac{1}{4}(A_{V}A_{U}W)^{v}}
={(R^{\nabla^{\prime}}_{U,V}W)^{v}}, \\ 
R^{{\tilde{\nabla}^{0}}}_{U^{v},V^{v}} X^{h_{\nabla}}
=&{-\frac{1}{4}(A_{U}A_{V}X)^{h_{\nabla}}
+\frac{1}{4}(A_{V}A_{U}X)^{h_{\nabla}}} 
={(R^{\nabla^{\prime}}_{U,V}X)^{h_{\nabla}}},\\ 
R^{{\tilde{\nabla}^{0}}}_{U^{v}, X^{h_{\nabla}}} V^{v}
=&{-\frac{1}{2}(J(H^{\nabla}_{U,V}J)X)^{h_{\nabla}}
-\frac{1}{4}(A_{X}A_{U}V)^{h_{\nabla}} -\frac{1}{4}(A_{U}A_{X} V)^{h_{\nabla}}}, \\
R^{{\tilde{\nabla}^{0}}}_{U^{v},X^{h_{\nabla}}}Y^{h_{\nabla}}
=& {\frac{1}{2} (J(H^{\nabla}_{X,Y} J)U)^{v}
   +\frac{1}{4} (A_{X}A_{Y}U)^{v} 
  +\frac{1}{4} (A_{U}A_{X} Y)^{v}}, \\
R^{{\tilde{\nabla}^{0}}}_{X^{h_{\nabla}},Y^{h_{\nabla}}}U^{v}
=&(R^{\nabla^{\prime}}_{X,Y}U)^{v}, \\
R^{{\tilde{\nabla}^{0}}}_{X^{h_{\nabla}},Y^{h_{\nabla}}}Z^{h_{\nabla}}
=&(R^{\nabla^{\prime}}_{X,Y} Z)^{h_{\nabla}} 
\end{align*}
for $X$, $Y$, $Z$, $U$, $V$, $W \in TN$, where $H^{\nabla}$ is the Hessian (the second covariant derivative) with respect to 
$\nabla$ and we have used Lemma~\ref{curvature_p}. Note that $(H^{\nabla}_{X,Y}J)(Z)=(H^{\nabla}_{X,Z}J)(Y)$ for 
all $X$, $Y$, $Z \in TN$, since $\nabla J$ is symmetric. 
Hence the flatness of $\nabla$ means that the Hessian of $J$ 
with respect to $\nabla$ is totally symmetric.    
By these equations, the Ricci tensor of ${\tilde{\nabla}^{0}}$ satisfies
\begin{align*}
Ric^{{\tilde{\nabla}^{0}}} (X^{h_{\nabla}},Y^{h_{\nabla}})
&=\frac{1}{2} \mathrm{Tr} J(H^{\nabla}_{X,Y} J)+\frac{1}{2} \mathrm{Tr} A_{X}A_{Y}, \\
Ric^{{\tilde{\nabla}^{0}}} (X^{h_{\nabla}},U^{v})
&=Ric^{{\tilde{\nabla}^{0}}} (U^{v},X^{h_{\nabla}})=0, \\
Ric^{{\tilde{\nabla}^{0}}} (U^{v},V^{v})
&=\frac{1}{2} \mathrm{Tr} J(H^{\nabla}_{U,V} J)
+\frac{1}{2} \mathrm{Tr} A_{U}A_{V}
\end{align*}
for $X$, $Y$, $U$, $V \in TN$. 
From $(\nabla J)J=-J(\nabla J)$, it holds that 
\[ \mathrm{Tr} J(H^{\nabla}_{X,Y} J)+\mathrm{Tr} A_{X}A_{Y}=0 \]
for all $X$, $Y \in TN$. 
Therefore the Obata connection of the manifolds obtained from our hypercomplex version of the c-map is Ricci flat.  
\end{proof}

\begin{remark}
{\rm 
The horizontal distribution on $M$ is integrable by (\ref{bracket}) and 
each leaf is totally geodesic with respect to the Obata connection $\tilde{\nabla}^{0}$, since  
${\tilde{\nabla}^{0}}_{X^{h_{\nabla}}}Y^{h_{\nabla}}
=(\nabla^{\prime} _{X}Y )^{h_{\nabla}}$ for $X$, $Y \in \Gamma(TN)$ . 
}
\end{remark}

\begin{remark}
{\rm 
In \cite[Theorem A]{F2}, a hypercomplex  structure was obtained on a neighborhood of the zero section of the tangent bundle of a complex manifold with a complex connection whose curvature is of type $(1,1)$. 
By contrast, our generalized rigid c-map gives a hypercomplex structure on the whole tangent bundle 
when the manifold is special complex.  
}
\end{remark}


\section{The c-projective structure on a projective special complex manifold}
\label{c-proj:sec}
\setcounter{equation}{0}

In this section, 
we discuss projective special complex manifolds and obtain the c-projective Weyl curvature 
of a canonically induced c-projective structure. 
Let $(N,J,\nabla,\xi)$ be a conical special complex manifold. 
Since $L_{\xi}J=0$ and $L_{J \xi}J=0$, 
we obtain a complex structure $\bar{J}$ 
on the quotient $\bar{N}:=N/\langle \xi, J\xi \rangle$ 
if $\bar{N}$ is a smooth manifold.

\begin{lemma}\label{der_conn_prime}
We have $L_{\xi} \nabla^{\prime}=0$ and $L_{J\xi} \nabla^{\prime}=0$. 
\end{lemma}

\begin{proof}
By Lemmas \ref{curvature_p} and \ref{sc:lem}, 
we have 
\begin{align*}
(L_{\xi} \nabla^{\prime})_{X}Y 
&=[\xi,  \nabla^{\prime}_{X}Y]
-\nabla^{\prime}_{[\xi,X]}Y-\nabla^{\prime}_{X}[\xi,Y] \\
&=\nabla^{\prime}_{\xi} \nabla^{\prime}_{X}Y - \nabla^{\prime}_{\nabla^{\prime}_{X}Y} \xi
-\nabla^{\prime}_{[\xi,X]}Y-\nabla^{\prime}_{X} \nabla^{\prime}_{\xi}Y
+\nabla^{\prime}_{X} \nabla^{\prime}_{Y} \xi \\
&=R^{\nabla^{\prime}}_{\xi,X}Y=0
\end{align*}
and 
\begin{align*}
(L_{J\xi} \nabla^{\prime})_{X}Y 
&=[J\xi,  \nabla^{\prime}_{X}Y]
-\nabla^{\prime}_{[J\xi,X]}Y-\nabla^{\prime}_{X}[J\xi,Y] \\
&=\nabla^{\prime}_{J\xi} \nabla^{\prime}_{X}Y - \nabla^{\prime}_{\nabla^{\prime}_{X}Y} J\xi
-\nabla^{\prime}_{[J\xi,X]}Y-\nabla^{\prime}_{X} \nabla^{\prime}_{J\xi}Y
+\nabla^{\prime}_{X} \nabla^{\prime}_{Y} J\xi \\
&=R^{\nabla^{\prime}}_{J\xi,X}Y=0
\end{align*}
for all $X$, $Y \in \Gamma(TN)$. 
\end{proof}

Recall \cite{Ishi} that a smooth curve $t\mapsto c(t)$ on a complex manifold $(M,J)$ 
is called {\cmssl $J$-planar with
respect to a connection $\nabla$ if $\nabla_{c'}c' \in \mathrm{span}\{ c', Jc'\}$.} 
We say that torsion-free complex connections $\nabla^{1}$ and $\nabla^{2}$ on a 
complex manifold $(M,J)$ are {\cmssl c-projectively related  \cite{CEMN}}   
if they have the same $J$-planar curves. It is known that 
$\nabla^{1}$ and $\nabla^{2}$ are c-projectively related
if and only if there exists a one-form $\theta$ on $M$ such that
\begin{align*}
\nabla^{1}_{X}Y
=&\nabla^{2}_{X} Y
+\theta(X)Y+\theta(Y)X
-\theta(JX)JY-\theta(JY)JX
\end{align*}
for $X$, $Y \in \Gamma(TM)$. See \cite{Ishi} for example. 
This defines an equivalence relation on the space of torsion-free complex connections on $M$.
The equivalence classes are called {\cmssl c-projective structures}.

\begin{definition}
We call the complex manifold $(\bar{N},\bar{J})$ a 
{\cmssl projective special complex manifold} if 
$p_{N}:(N,J,\nabla,\xi) \to (\bar{N},\bar{J})$ is a principal $\mathbb{C}^{\ast}$-bundle, where  
the principal $\mathbb{C}^{\ast}$-action is generated by the holomorphic vector field 
$\xi-\sqrt{-1} J \xi$.
\end{definition}

Note that a projective special K{\"a}hler manifold is   
a K\"ahler quotient of a conical special K{\"a}hler manifold. 
Similarly, a projective special complex manifold 
carries an induced c-projective structure as follows.

\begin{proposition}\label{projective_str}
Any projective special complex manifold $(\bar{N},\bar{J})$ carries a canonical c-projective structure.
\end{proposition}


\begin{proof}
Consider a connection form 
$\hat{\alpha}=\alpha-\sqrt{-1} (\alpha \circ J)$ of 
type $(1,0)$ on the principal\linebreak[3] $\mathbb{C}^{\ast}$-bundle
$p_{N}:N \to \bar{N}$.  (Note that any $\mathbb{C}^*$-invariant real one-form $\alpha$ such that $\alpha (\xi) =1$ is the real part of such a connection.) 
We have $TN=\mathrm{Ker} \, \hat{\alpha} \oplus \langle \xi, J \xi \rangle$, where 
$\mathrm{Ker} \, \hat{\alpha}$ is $J$-invariant.  
We denote the $\hat{\alpha}$-horizontal lift of $X \in \Gamma(T \bar{N})$ 
by $X^{h_{\alpha}}$. 
By Lemma \ref{der_conn_prime}, we can define  
$\bar{\nabla}^{\prime \alpha}$ by 
\begin{align}\label{ind_conn}
\bar{\nabla}^{\prime \alpha}_{X} Y
=p_{N \ast} (\nabla^{\prime}_{X^{h_{\alpha}}} Y^{h_{\alpha}})
\end{align}
for $X$, $Y \in \Gamma(T \bar{N})$. 
We claim that $\bar{\nabla}^{\prime \alpha} \bar{J}=0$. 
In fact, using that $JY^{h_{\alpha}}=(\bar{J} Y)^{h_{\alpha}}$  for $Y\in T\bar{N}$ we have 
\begin{align*}
\bar{\nabla}^{\prime \alpha}_{X} (\bar{J}Y)
=p_{N \ast} (\nabla^{\prime}_{X^{h_{\alpha}}} J Y^{h_{\alpha}})
=p_{N \ast} J(\nabla^{\prime}_{X^{h_{\alpha}}} Y^{h_{\alpha}})
=\bar{J} p_{N \ast} (\nabla^{\prime}_{X^{h_{\alpha}}} Y^{h_{\alpha}}). 
\end{align*}
To show that the c-projective structure $[\bar{\nabla}^{\prime \alpha}]$ does not 
depend on $\alpha$, we consider another 
connection form $\hat{\beta}=\beta-\sqrt{-1} (\beta \circ J)$ of 
type $(1,0)$. 
Then there exist one-forms $\theta_{0}$ and $\theta_{1}$ on $\bar{N}$ 
such that 
\begin{align*}\label{diff_conn}
\hat{\beta}-\hat{\alpha}=(p_{N}^{\ast} \theta_{0}) 
+(p_{N}^{\ast} \theta_{1}) \sqrt{-1}.
\end{align*}
On the other hand, 
we can write $X^{h_{\alpha}} - X^{h_{\beta}}=a \xi +b J \xi$ for some 
functions $a$, $b$ on $N$. It is easy to see that
\[ a=\theta_{0}(X)\circ p_{N}, \,  
b=-\theta_{0}(\bar JX){\circ p_{N}}, \, \theta_{1}=-\theta_{0} \circ \bar J \]
for $X \in T\bar N$. 
By the definition (\ref{ind_conn}) of the induced connection on $\bar{N}$, 
we have
\begin{align*}
\bar{\nabla}^{\prime \alpha}_{X} Y
=&p_{N \ast} (\nabla^{\prime}_{X^{h_{\alpha}}} Y^{h_{\alpha}}) \\
=&p_{N \ast} (\nabla^{\prime}_{X^{h_{\beta}} +\theta_{0}(X)\xi-\theta_{0}(\bar JX) J\xi} 
(Y^{h_{\beta}} +\theta_{0}(Y)\xi-\theta_{0}(\bar JY) J \xi)) \\
=&p_{N \ast} ( \nabla^{\prime}_{X^{h_{\beta}}} Y^{h_{\beta}} 
    +\nabla^{\prime}_{X^{h_{\beta}}} \theta_{0}(Y)\xi - \nabla^{\prime}_{X^{h_{\beta}}} \theta_{0}
(\bar JY) J\xi \\
  &\hspace{1cm} +\theta_{0}(X) ( \nabla^{\prime}_{\xi}Y^{h_{\beta}} 
                                   +\nabla^{\prime}_{\xi} \theta_{0}(Y)\xi- \nabla^{\prime}_{\xi}\theta_{0}
(\bar JY)J \xi ) \\
  &\hspace{1cm} -\theta_{0}(\bar JX) ( \nabla^{\prime}_{J \xi}Y^{h_{\beta}}
                                   +\nabla^{\prime}_{J \xi}\theta_{0}(Y)\xi-\nabla^{\prime}_{J \xi}\theta_{0}
(\bar JY) J \xi) \\
=&\bar{\nabla}^{\prime \beta}_{X} Y
+\theta_{0}(Y)X+\theta_{0}(X)Y
-\theta_{0}(\bar{J}Y)\bar{J}X-\theta_{0}(\bar{J}X)\bar{J}Y
\end{align*}
for $X$, $Y \in \Gamma(T \bar{N})$, which means that 
$\bar{\nabla}^{\prime \alpha}$
and $\bar{\nabla}^{\prime \beta}$ are c-projectively related. 
Here we write $\theta_{0}(X)$ for 
$\theta_{0}(X)\circ p_{N}$ etc. 
\end{proof}

We denote the induced c-projective structure given in Proposition~\ref{projective_str}
by $\mathcal{P}_{\bar{\nabla}^{\prime}}$ (without a label $\alpha$). 
Next we prove that the c-projective Weyl curvature of $\mathcal{P}_{\bar{\nabla}^{\prime}}$
is of type $(1,1)$ (see Theorem \ref{Weyl_bar}).

Note that $\xi, J \xi$ are the fundamental vector fields generated by $1,\sqrt{-1}\in \mathbb{C} = \mathrm{Lie}\, \mathbb{C}^*$, respectively.   
Recall that $A=\nabla J$ and $A_\xi = A_{J\xi}=0$. 
We also have that $L_\xi A =0$, since $L_\xi \nabla=0$ and $L_\xi J =0$.

\begin{lemma}\label{a1_der} 
$L_{J\xi}\nabla = A$, $L_{J\xi}A=-2JA$ and $L_{J\xi}(JA) = 2 A$. 
\end{lemma}

Let $\eta$ be a connection form on the principal bundle $p_{N}:N \rightarrow \bar N$.  
As before, we assume that $\mathrm{Ker}\, \eta $ is 
$J$-invariant or, equivalently, that $\eta$ is of type $(1,0)$
 (but not necessarily holomorphic). 
Using $\eta$ we can project the connection 
$\nabla^{\prime}$ on $N$ to a connection $\bar{\nabla}^{\prime \, \eta}$ on 
$\bar{N}$, which is complex with respect to $\bar{J}$, as shown in the proof of Proposition~\ref{projective_str}.  
Note that the quotient $p_{N}:(N,\nabla^{\prime}) \to (\bar{N}, \bar{\nabla}^{\prime \, \eta})$
is an affine submersion with the horizontal distribution $\mathcal{H}:=\mathrm{Ker}\, \eta$ 
(in the sense defined in \cite{AH}).  
From now on the $\eta$-horizontal lift of $X \in T\bar{N}$ is denoted by $\tilde{X}$. 
Note that our sign convention for the curvature tensor is different from the one in \cite{AH}. 
Let $h:TN \to \mathcal{H}$ and $v:TN \to \mathcal{V}$ be the projections with respect to 
the decomposition $TN= \mathcal{H} \oplus \mathcal{V}$, where $\mathcal{V}=\mathrm{Ker}\, p_{N \, \ast}$. 
We define the fundamental tensors $\mathcal{A}^{\nabla^{\prime}}$ and 
$\mathcal{T}^{\nabla^{\prime}}$ by 
\[ \mathcal{A}^{\nabla^{\prime}}_{E} F
=v (\nabla^{\prime}_{hE} hF) + h(\nabla^{\prime}_{hE} vF) \]
and
\[ \mathcal{T}^{\nabla^{\prime}}_{E} F
=h (\nabla^{\prime}_{vE} vF) + v(\nabla^{\prime}_{vE} hF) \]
for $E$, $F \in \Gamma(T N)$. 

\begin{lemma}
We have $ \mathcal{T}^{\nabla^{\prime}}=0$, 
$\mathcal{A}^{\nabla^{\prime}}_{X}{\xi} = X$ and $\mathcal{A}^{\nabla^{\prime}}_{X} J\xi =JX$
for any horizontal vector $X$. 
\end{lemma}

Let $a$ and $b$ be $(0,2)$-tensors defined by  
\[ \mathcal{A}^{ \nabla^{\prime}}_{X} Y=a(X,Y)\xi + b(X,Y)J\xi \]
for horizontal vectors $X$ and $Y$. 
Since $\nabla^{\prime}$ and the projections $v$, $h$ are $\mathbb{C}^{\ast}$-invariant, 
$\mathcal{A}^{\nabla^{\prime}}$ is $\mathbb{C}^{\ast}$-invariant, and hence, 
$a=p_{N}^{\ast} \bar{a}$ and $b=p_{N}^{\ast} \bar{b}$ for some tensors $\bar{a}$ and $\bar{b}$ on $\bar{N}$. 
For any $(0,2)$-tensor $k$ on a complex manifold with a complex structure $J$, 
define the $(0,2)$-tensor $k_{J}$ by $k_{J}(X,Y):=k(X,JY)$. 

\begin{lemma}\label{vert}
We have 
\begin{align*}
v((\nabla^{\prime}_{\tilde{X}} J) \tilde{Y} )
&=\left( \bar{a}(X,\bar{J}Y) + \bar{b}(X,Y) \right) \xi 
+\left( \bar{b}(X,\bar{J} Y) - \bar{a}(X,Y) \right) J \xi 
\end{align*}
for $X$, $Y \in T \bar{N}$.
\end{lemma}

\begin{lemma}\label{A_eq}
We have 
$\bar{b}(X,Y)=-\bar{a}(X,\bar{J}Y)=-\bar{a}_{\bar{J}}(X,Y)$ for tangent vectors $X$ and $Y$ on $\bar{N}$. 
Consequently, the fundamental tensor $\mathcal{A}^{\nabla^{\prime}}$ satisfies 
\begin{align}
\mathcal{A}^{ \nabla^{\prime}}_{ \tilde{X}} \tilde{Y} 
&=\bar{a}(X,Y)\xi - \bar{a}_{\bar{J}}(X,Y) J\xi 
%
\end{align}
for tangent vectors $X$, $Y$ on $\bar{N}$. 
\end{lemma}

\begin{proof}
By $\nabla^{\prime}J=0$ and Lemma \ref{vert}, we have the conclusion. 
\end{proof}

Let $(r,\theta)$ be the polar coordinates with respect to a (smooth) local trivialization 
$p_{N}^{-1}(\bar{U}) \cong \bar{U} \times \mathbb{C}^*$ of the 
principal $\mathbb{C}^*$-bundle $p_{N} : N \rightarrow \bar{N}$ 
such that $\xi=r \partial/ \partial r$ and 
$J\xi = \partial/\partial \theta$. 
A principal connection $\eta$ is locally given by 
\[ \eta := \eta_{1} \otimes 1 +\eta_{2} \otimes \sqrt{-1} 
= p_{N}^{\ast}(\gamma_{1} \otimes 1+ \gamma_{2} \otimes \sqrt{-1}) 
+ \left(\frac{dr}{r} \otimes 1 +d \theta \otimes \sqrt{-1} \right) \]
for a $\mathbb{C}$-valued one-form $\gamma_{1} \otimes 1+ \gamma_{2} \otimes \sqrt{-1}$
on $\bar{U} \subset \bar{N}$. 
For each local trivialization 
$p_{N}^{-1}(\bar{U}) \cong \bar{U} \times \mathbb{C}^*$,  
we set 
\[ B:=e^{2 \theta J} A \,\,\,  (e^{2 \theta J}=(\cos 2 \theta) \mathrm{id}+(\sin 2 \theta)J). \]


The symmetric $(1,2)$-tensor field $B$ is defined {\it locally} and 
$B$ is projectable by Lemma \ref{a1_der}, i.e.\ horizontal (i.e.\ $B_\xi = B_{J\xi}=0$) and $\mathbb{C}^*$-invariant. 
Therefore we obtain an induced {\it locally} defined symmetric tensor field $\bar{B}$ on 
$\bar{N}$.

\begin{lemma}\label{square_B}
The tensor $B^2 : (X,Y) \mapsto B_X\circ B_Y$ is a globally defined tensor field on $N$, in particular, 
$[B,B]$ is so. 
As a consequence, we have the {\it globally} defined tensor fields $\bar{B}^{2}$ and $[\bar{B},\bar{B}]$ on $\bar{N}$. 
\end{lemma}

\begin{proof}
It follows from $B^{2}=A^{2}$. 
\end{proof}

For a $(0,2)$-tensor $a$ and a $(1,1)$-tensor $K$, we define an $\mathrm{End}(TN)$-valued $2$-form 
$a \wedge K$ by 
\[  (a \wedge K)_{X,Y}Z=a(X,Z)KY -a(Y,Z)KX \]
for tangent vectors $X$, $Y$ and $Z$.

\begin{proposition}\label{curv} 
The curvature $R^{ \bar{\nabla}^{\prime \, \eta} }$ of  $\bar{\nabla}^{\prime \, \eta}$ is of the form 
\begin{align*}
R^{ \bar{\nabla}^{\prime \, \eta}} 
=& 
-\frac{1}{4}[\bar{B},\bar{B}]
 +2 \bar{a}^{a} \otimes Id - 2 (\bar{a}_{\bar{J}})^{a} \otimes \bar{J} 
+ \bar{a} \wedge Id - \bar{a}_{\bar{J}} \wedge \bar{J},  
\end{align*}
where  
$(\, \cdot \,)^{a}$ denotes anti-symmetrization. Moreover we have 
$d \gamma_{1} =-2 \bar{a}^{a}$ and 
$d \gamma_{2} =2 (\bar{a}_{\bar{J}})^{a}$. 
\end{proposition}

\begin{proof} 
By the fundamental equation for an affine submersion \cite{AH}, we have
\[ (R^{\bar{\nabla}^{\prime \, \eta}}_{X,Y} Z){\,}^{\tilde{}}
= h(R^{\nabla^{\prime}}_{\tilde{X},\tilde{Y}} \tilde{Z})+
h(\nabla^{\prime}_{v[\tilde{X},\tilde{Y}]} \tilde{Z}) 
+\mathcal{A}^{\nabla^{\prime}}_{\tilde{Y}} \mathcal{A}^{\nabla^{\prime}}_{\tilde{X}} \tilde{Z}
-\mathcal{A}^{\nabla^{\prime}}_{\tilde{X}} \mathcal{A}^{\nabla^{\prime}}_{\tilde{Y}} \tilde{Z}
\]
for $X$, $Y$, $Z \in \Gamma(T \bar{N})$. 
Since
\begin{align*}
v[\tilde{X},\tilde{Y}] &=\eta_{1}([\tilde{X},\tilde{Y}])\xi + \eta_{2} ([\tilde{X},\tilde{Y}]) J \xi \\
                            &=-(d \eta_{1})(\tilde{X},\tilde{Y})\xi - (d \eta_{2}) (\tilde{X},\tilde{Y}) J \xi \\
                            &=-(d \gamma_{1})(X,Y) \xi -(d \gamma_{2})(X,Y) J \xi,  
\end{align*}
we have
\begin{align*}
h(\nabla^{\prime}_{v[\tilde{X},\tilde{Y}]} \tilde{Z}) 
&=h(\nabla^{\prime}_{\tilde{Z}} v[\tilde{X},\tilde{Y}])  \\
&=h(\nabla^{\prime}_{\tilde{Z}} (-(d \gamma_{1})(X,Y) \xi -(d \gamma_{2})(X,Y) J \xi ) ) \\
&=-(d \gamma_{1})(X,Y) \tilde{Z} -(d \gamma_{2})(X,Y) (\bar{J}Z)^{\tilde{}}. 
\end{align*}
Moreover, by 
\[ \mathcal{A}^{\nabla^{\prime}}_{\tilde{X}}\tilde{Y}- \mathcal{A}^{\nabla^{\prime}}_{\tilde{Y}}\tilde{X} 
=v[\tilde{X},\tilde{Y}] =-d \gamma_{1}(X,Y) \xi -d \gamma_{2}(X,Y) J \xi, \]
we have $d \gamma_{1} =-2 \bar{a}^{a}$ and 
$d \gamma_{2} =2 (\bar{a}_{\bar{J}})^{a}$. 
\end{proof}

Now we set $\dim N=2(n+1)$.
By Proposition \ref{curv} and $\mathrm{Tr} \bar{B}_{X}=0$ for all 
$X \in T\bar{N}$, we obtain 
\begin{align}
Ric^{ \bar{\nabla}^{\prime \, \eta}}(Y,Z)
=&\frac{1}{4}\mathrm{Tr} \bar{B}_{Y}\bar{B}_{Z}
   +(\bar{a}(Z,Y)-\bar{a}(Y,Z)) \label{ricci_proj_a} \\
  &-(\bar{a}(\bar{J}Y,\bar{J}Z)+\bar{a}(Y,Z)) 
    -2n \bar{a}(Y,Z)+\bar{a}(Y,Z) 
    -\bar{a}(\bar{J}Y,\bar{J}Z) \nonumber \\
=&\frac{1}{4}\mathrm{Tr} \bar{B}_{Y}\bar{B}_{Z}
  -(2n+1) \bar{a}(Y,Z)+\bar{a}(Z,Y)  \nonumber \\
  &-\bar{a}(\bar{J}Y,\bar{J}Z)-\bar{a}(\bar{J}Z,\bar{J}Y). \nonumber
\end{align}

We define a $(0,2)$-tensor $P^{D}$ on a complex manifold 
$(M,J)$, which is called 
the {\cmssl Rho tensor} of a connection $D$,  by 
\begin{align*}
P^{D}(X,Y) =& \frac{1}{m+1} \left(
Ric^{D}(X,Y)+\frac{1}{m-1} 
\left( (Ric^{D})^{s}(X,Y) -(Ric^{D})^{s}(JX, JY) \right) \right), 
\end{align*}
for $X$, $Y \in T M$, where $2m=\dim M \geq 4$, 
$Ric^{D}$ is the Ricci tensor of $D$ and 
$(\, \cdot \,)^{s}$ is the symmetrization of a $(0,2)$-tensor. 
The {\cmssl c-projective Weyl curvature} $W^{c, [\bar{D}]}$ 
of a c-projective structure $[\bar{D}]$ is given by 
\begin{align}\label{c_weyl}
W^{c,[\bar{D}]} =& R^{\bar{D}} +(P^{\bar{D}})^{a} \otimes Id 
       -(P^{\bar{D}}_{\bar{J}})^{a} \otimes \bar{J}+\frac{1}{2} P^{\bar{D}} \wedge Id -\frac{1}{2} P^{\bar{D}}_{\bar{J}} \wedge \bar{J}. 
\end{align}
See \cite{CEMN}.  
We shall compute the c-projective Weyl curvature of $[ \bar{\nabla}^{\prime \, \eta}]$. 
From (\ref{ricci_proj_a}) it holds
\begin{align*}
(Ric^{ \bar{\nabla}^{\prime \, \eta}})^{s}(Y,Z)
&=\frac{1}{4}\mathrm{Tr} \bar{B}_{Y}\bar{B}_{Z}
   -2n \bar{a}^{s}(Y,Z)-2 \bar{a}^{s}(\bar{J}Y, \bar{J}Z), \\ 
   (Ric^{ \bar{\nabla}^{\prime \, \eta}})^{s}(\bar{J}Y,\bar{J}Z)
&=\frac{1}{4}\mathrm{Tr} \bar{B}_{Y}\bar{B}_{Z}
   -2n \bar{a}^{s}(\bar{J}Y,\bar{J}Z)-2 \bar{a}^{s}(Y, Z)
   \end{align*}
and hence  
\begin{align*}
(Ric^{ \bar{\nabla}^{\prime \, \eta}})^{s}(Y,Z)
-(Ric^{ \bar{\nabla}^{\prime \, \eta}})^{s}(\bar{J}Y,\bar{J}Z)
&=-2(n-1) \left( \bar{a}^{s}(Y,Z)- \bar{a}^{s}( \bar{J}Y, \bar{J}Z) \right). 
\end{align*}
From these equations, it follows that 
\begin{align*}
(n+1)  P^{ \bar{\nabla}^{\prime \, \eta}}(Y,Z) 
=& \frac{1}{4}\mathrm{Tr} \bar{B}_{Y}\bar{B}_{Z}
  -(2n+1) \bar{a}(Y,Z)+\bar{a}(Z,Y) 
  -\bar{a}(\bar{J}Y,\bar{J}Z)-\bar{a}(\bar{J}Z,\bar{J}Y) \\
  &-2 (\bar{a}^{s}(Y,Z)- \bar{a}^{s}( \bar{J}Y, \bar{J}Z) ) \\
=&\frac{1}{4}\mathrm{Tr} \bar{B}_{Y}\bar{B}_{Z}
  -(2n+1) \bar{a}(Y,Z)+\bar{a}(Z,Y) -(\bar{a}(Y,Z)+\bar{a}(Z,Y)) \\
 =& \frac{1}{4}\mathrm{Tr} \bar{B}_{Y}\bar{B}_{Z} 
     -2(n+1)\bar{a}(Y,Z). 
\end{align*}
Setting $\bar{\cal B}(Y,Z)=\mathrm{Tr} \bar{B}_{Y}\bar{B}_{Z}$, 
which is a symmetric, $\bar{J}$-hermitian globally defined 
$(0,2)$-tensor on $\bar{N}$, 
we have
\begin{align}\label{bar_a}
\bar{a}=\frac{1}{8(n+1)} \bar{\cal B}
                 -\frac{1}{2} P^{ \bar{\nabla}^{\prime \, \eta}}. 
\end{align}
Therefore the coefficients of the curvature form $d\eta=d\gamma_1 + \sqrt{-1}d\gamma_2= -2\bar{a}^a +2\sqrt{-1}(\bar{a}_{\bar J})^a$ are determined by 
\begin{align}
\bar{a}^{a} &=-\frac{1}{2} (P^{\bar{\nabla}^{\prime \, \eta}})^{a} \left(=-\frac{1}{2(n+1)}
 (Ric^{ \bar{\nabla}^{\prime \, \eta}})^{a} \right), \label{a_1} \\
(\bar{a}_{\bar{J}})^{a} & = \frac{1}{8(n+1)} \bar{\cal B}_{\bar{J}}
-\frac{1}{2} (P^{ \bar{\nabla}^{\prime \, \eta}}_{\bar{J}})^{a}
\left( =  
\frac{1}{8(n+1)} \bar{\cal B}_{\bar{J}}
-\frac{1}{2(n+1)} (Ric^{ \bar{\nabla}^{\prime \, \eta}}_{\bar{J}} )^{a}
 \right) . \label{a_2}
\end{align}

By the above calculations we arrive at the following theorem.   
\begin{theorem}\label{Weyl_bar}
Let $(N,J,\n , \xi )$ be a conical special complex manifold 
which is the total space of a (holomorphic) principal 
$\mathbb{C}^*$-bundle $p_{N} : N \rightarrow \bar{N}$,  
the base of which is a projective special complex manifold $\bar{N}$ with 
$\dim \bar{N} =2n \geq 4$. 
The c-projective Weyl curvature $W^{c,\mathcal{P}_{\bar{\nabla}^{\prime}} }$ of the 
canonically induced c-projective structure $\mathcal{P}_{\bar{\nabla}^{\prime}}$ 
is given by 
\begin{align*}
W^{c,\mathcal{P}_{\bar{\nabla}^{\prime}} }
=& -\frac{1}{4} [\bar{B},\bar{B}] 
-\frac{1}{4(n+1)} \bar{\cal B}_{\bar{J}} \otimes \bar{J} 
+\frac{1}{8(n+1)} \bar{\cal B} \wedge \mathrm{Id}
    -\frac{1}{8(n+1)} \bar{\cal B}_{\bar{J}} \wedge \bar{J}. 
\end{align*}
In particular, $W^{c,\mathcal{P}_{\bar{\nabla}^{\prime}} }_{\bar{J}(\, \cdot \, ),\bar{J}(\, \cdot \, )}=W^{c,\mathcal{P}_{\bar{\nabla}^{\prime}} }$, that is, 
$W^{c,\mathcal{P}_{\bar{\nabla}^{\prime}} }$ is of type $(1,1)$ 
as an $\mathrm{End}(T\bar{N})$-valued two-form. 
\end{theorem}

\begin{proof} 
Take a principal connection $\eta$ of type $(1,0)$. 
By Proposition \ref{projective_str}, the canonically induced c-projective structure is 
$[\bar{\nabla}^{\prime \, \eta}]$.   
From Proposition \ref{curv}, equation (\ref{c_weyl}) and the symmetry of $\bar{\cal B}$,  
it holds that
\begin{align*}\label{a_3}
W^{c, [\bar{\nabla}^{\prime \, \eta}]} 
=& -\frac{1}{4} [\bar{B},\bar{B}] 
-\frac{1}{4(n+1)} \bar{\cal B}_{\bar{J}} \otimes \bar{J} 
+\frac{1}{8(n+1)} \bar{\cal B} \wedge \mathrm{Id}
    -\frac{1}{8(n+1)} \bar{\cal B}_{\bar{J}} \wedge \bar{J}.
\end{align*}
Since $\bar{\cal B}_{\bar{J}}$, $[\bar{B},\bar{B}]$ and 
$\bar{\cal B} \wedge \mathrm{Id}
-\bar{\cal B}_{\bar{J}} \wedge \bar{J}$ are of type $(1,1)$, 
$W^{c,\mathcal{P}_{\bar{\nabla}^{\prime}} }$ is of type $(1,1)$.
\end{proof}
The following corollary is a direct consequence of Theorem \ref{Weyl_bar}. 

\begin{corollary}
Any complex manifold $(\bar{N},\bar{J})$ with 
a c-projective structure $\mathcal{P}$ such that 
$W^{c,\mathcal{P}}$ is not of type $(1,1)$ 
can not be realized as 
a projective special complex manifold whose canonical c-projective structure 
is $\mathcal{P}$. 
\end{corollary}

\section{A generalization of the supergravity c-map}
\setcounter{equation}{0}

The supergravity c-map associates a (pseudo-)quternionic K{\"a}hler manifold   
with any projective special K{\"a}hler manifold. In this section, we give 
a generalization of the supergravity c-map by using the results in previous sections. 
Let $(N,J,\nabla, \xi)$ be a conical special complex manifold and set $Z:=J \xi$.

\begin{lemma}\label{rotation_eq}
$2Z^{h_{\nabla}}$ is a rotating vector field on $TN$. 
\end{lemma}

\begin{proof}
Since $L_{Z}J=0$ and 
$\nabla_Z J=0$ (cf.\ Lemma \ref{sc:lem}), 
we have $L_{Z^{h_{\nabla}}} I_{1}=0$. Moreover
we have
\begin{align*}
(L_{Z^{h_{\nabla}}}I_{2})(X^{h_{\nabla}}+Y^{v}) 
&=[Z,Y]^{h_{\nabla}}-(\nabla_{Z} X)^{v}-(\nabla_{Z} Y)^{h_{\nabla}}+[Z,X]^{v} \\
&=-(\nabla_{Y}Z)^{h_{\nabla}}-(\nabla_{X}Z)^{v} \\
&=-(JY)^{h_{\nabla}}-(JX)^{v} \\
&=-I_{3}(X^{h_{\nabla}}+Y^{v})
\end{align*}
for all  $X$, $Y \in \Gamma(TN)$. 
\end{proof}

\begin{remark}
{\rm 
By the equations for ${\tilde{\nabla}^{0}}$ in the proof of Theorem~\ref{ricci_flat}, we have  
\begin{align*}
{\tilde{\nabla}^{0}}_{X^{h_{\nabla}}} \xi^{h_{\nabla}}
&=(\nabla^{\prime}_{X} \xi)^{h_{\nabla}}
=\left( \nabla_{X} \xi -\frac{1}{2} JA_{X} \xi \right) ^{h_{\nabla}}
=X^{h_{\nabla}},\\
{\tilde{\nabla}^{0}}_{X^{v}} \xi^{h_{\nabla}}
&=-\frac{1}{2} (J A_{\xi} X)^{v}=0
\end{align*}
for $X \in TN$, when $(N,J,\nabla , \xi)$ is a conical special complex manifold.   
}
\end{remark}



We have the following theorem.

\begin{theorem}[Generalized supergravity c-map]\label{gene_sup_cmap}
Let $(N,J,\nabla,\xi)$ be a $2(n+1)$-dimensional conical special complex manifold. 
Let $\Theta$ be a closed two-form on $M=TN$ such that $L_{Z^{M}} \Theta=0$, 
where $Z^{M}=2 Z^{h_{\nabla}}$.  
Consider a $\mathrm{U}(1)$-bundle 
$\pi : P \to M$ over $M$ 
and $\eta$ a connection form whose curvature form is 
\[ d \eta
   =\pi^{\ast} \left( \Theta-\frac{1}{2} d ((\iota_{Z^{M}} \Theta) \circ I_{1}) \right). \]
Let $f$ be a smooth function on $M$ such that $df = -\iota_{Z^{M}} \Theta$ and 
$f_{1}:=f-(1/2)\Theta(Z^{M},I_{1}Z^{M})$ does nowhere vanish. 
If $\tilde{\pi}:\tilde{M} \to \hat{M}$ and $\hat{\pi}:\hat{M} \to \bar{M}$ are submersions, 
we have an assignment from a $2n$-dimensional projective special complex manifold 
$(\bar{N},\bar{J}, \mathcal{P}_{\bar{\nabla}^{\prime}})$ 
whose c-projective Weyl curvature is of type $(1,1)$ 
to a $4(n+1)$-dimensional quaternionic manifold 
\[ \bar{M}(=\overline{TN})=\mathcal{C}_{(P,\eta)}(M, \langle I_{1},I_{2},I_{3} \rangle,Z^{M}, f, 
\Theta) 
/ \mathcal{D} \] 
foliated by $(2n+4)$-dimensional leaves 
such that $\bar{N}$ coincides with the space of its leaves.  
\end{theorem}

\begin{proof}
By Theorem \ref{gene_hq_corresp}, Lemma \ref{rotation_eq} 
and Proposition \ref{projective_str}, we have an assignment from a $2n$-dimensional projective special complex manifold 
$(\bar{N},\bar{J}, \mathcal{P}_{\bar{\nabla}^{\prime}})$ 
to a $4(n+1)$-dimensional quaternionic manifold 
$\overline{TN}$. By virtue of Theorem \ref{Weyl_bar}, the c-projective Weyl curvature of 
$\mathcal{P}_{\bar{\nabla}^{\prime}}$
is of type $(1,1)$. 
Next we give a foliation on $\overline{TN}$ whose leaves space is $\bar{N}$. 
Set $\mathcal{L}:=\mathcal{V} \oplus \langle \xi^{h_{\nabla}}, Z^{h_{\nabla}} \rangle$, where $\mathcal{V}$ 
is the vertical distribution of $T(TN) \to TN$. The distribution $\mathcal{L}$ is 
$Z^{M}=2Z^{h_{\nabla}}$-invariant and integrable by (\ref{bracket}). 
Therefore each leaf $L$ of $\mathcal{L}$ 
is a $Z^{M}=2Z^{h_{\nabla}}$-invariant submanifold of $TN$. 
Consider the pull-back $\iota^{\#}P$ of $P$ by the inclusion $\iota:L \to TN$ 
with the bundle map $\iota_{\#}:\iota^{\#}P \to P$
and $\tilde{L}:=\mathbb{H}^{\ast} \times \iota^{\#}P$. 
Since $V_{1}$ is tangent to $\tilde{L}$, then  $\hat{L}:=\tilde{L}/\langle V_{1} \rangle$ 
is a submanifold $\hat{M}$. Moreover $V$, $\hat{I}_{1}(V)$, $\hat{I}_{2}(V)$, $\hat{I}_{3}(V)$ 
are tangent to  $\hat{L}$ because $V$ is induced by $e_{0}^{R}$.  
Taking the quotient again, we obtain a submanifold  
$\bar{L}:=\hat{L}/ \langle V, \hat{I}_{1}(V), \hat{I}_{2}(V), \hat{I}_{3}(V) \rangle$  
on a quaternionic manifold $\overline{TN}$. 
Therefore the quaternionic manifold $\overline{TN}$ is 
foliated by $(2n+4)$-dimensional leaves 
such that the space of its leaves $\bar{L}$ is the projective special complex manifold $\bar{N}$. 
\end{proof}

\begin{remark}
{\rm
If we assume that $Z_{1}=(Z^{M})^{h_{\eta}}+f_{1}X_{P}$ generates 
a free $\mathrm{U}(1)$-action on $P$ instead of assuming that  
$\tilde{\pi}:\tilde{M} \to \hat{M}$ and $\hat{\pi}:\hat{M} \to \bar{M}$ 
are submersions, 
we obtain the same result as in Theorem \ref{gene_sup_cmap} (see Theorem \ref{swan_twist}). 
}
\end{remark}

\begin{remark}\label{BC:rem}
{\rm Bor{\' o}wka and Calderbank have given a construction of a quaternionic manifold from a complex manifold of half 
the dimension with a c-projective structure, known as the quaternionic Feix-Kaledin construction \cite{BC}. 
Their construction generalizes the original construction \cite{F,K}, which yields a hyper-K\"ahler structure 
on a neighborhood of the zero setion of any K\"ahler manifold. 
They also point out that 
this construction is a generalization of \cite[Theorem A]{F2} (see \cite[Proposition 5.4]{BC}). 
More precisely, the initial data of the quaternionic Feix-Kaledin construction are 
a complex manifold with a c-projective structure of type $(1,1)$ and 
a complex line bundle with a connection of type $(1,1)$. Note that
this construction is different from our generalization of the supergravity c-map, in which  
the real dimension of the quaternionic manifold $\overline{TN}$ is related to the real dimension of the projective special complex manifold $\bar N$ by 
$\dim (\overline{TN}) = 2\dim (\bar{N}) + 4$.   
}
\end{remark}

We consider a conical special complex manifold $(N,J,\nabla,\xi)$, which we endow now 
with an additional structure. Let $\psi$ be a $J$-hermitian, $\nabla$-parallel two-form on 
$(N,J,\nabla,\xi)$. 
We consider a function $\mu=(1/2) \psi(\xi,J \xi)$ on $N$. Then we see 
$d \mu=-\iota_{Z} \psi$. 
Set 
\begin{align}
\Theta &=-\pi_{TN}^{\ast} \psi \label{a1},\\
f &=-2\pi_{TN}^{\ast} \mu +c \label{a2}
\end{align}
for some constant $c$. 
Then it holds that 
\[ df=-\iota_{Z^{M}} \Theta, \, 
f_{1}=f-\frac{1}{2} \Theta(Z^{M},I_{1}Z^{M})=2 \pi_{TN}^{\ast} \mu +c, \]
where $\pi_{TN}:TN \to N$ is the bundle projection.
In fact, we have 
\begin{align*}
df &= -2d  (\pi_{TN}^{\ast} \mu) 
    = 2 \pi_{TN}^{\ast} (\iota_{Z} \psi) 
    = - \iota_{Z^{M}} \Theta
\end{align*}
and 
\begin{align*}
f_{1}&=f-\frac{1}{2} \Theta(Z^{M},I_{1}Z^{M}) \\
     &=-\psi(\xi,J\xi) \circ \pi_{TN}-2 \Theta(Z^{h_{\nabla}},I_{1}Z^{h_{\nabla}})+c\\
     &=\psi(J\xi,\xi)\circ \pi_{TN} +c 
     =2 \pi_{TN}^{\ast} \mu +c.
\end{align*}

\begin{corollary}\label{with_two_form}
Let $(N,J,\nabla,\xi)$ be a $2n$-dimensional conical special complex 
manifold and 
$\psi$ a $J$-hermitian, $\nabla$-parallel two-form on $N$. 
Consider a $\mathrm{U}(1)$-bundle 
$\pi : P \to M$ over $M=TN$ 
and $\eta$ a connection form whose curvature form is 
\[ d \eta
   =(\pi_{TN} \circ \pi)^{\ast} \psi.  \]
If 
$\tilde{\pi}:\tilde{M} \to \hat{M}$ and $\hat{\pi}:\hat{M} \to \bar{M}$ 
are submersions and 
$\mu^{-1}(-c/2) =\emptyset$, then the generalized supergravity c-map of Theorem \ref{gene_sup_cmap} can be specialized to this setting such that the data $\Theta$ and $f$ are related to $\psi$ by equations (\ref{a1}) and (\ref{a2}). 
\end{corollary}

\begin{proof}
By a straightforward calculation, we have $d ((\iota_{Z} \psi) \circ J)= 2 \psi$.
Then it is easy to check 
\begin{align*}
 d \eta &= (\pi_{TN} \circ \pi)^{\ast} \psi \\
           &=(\pi_{TN} \circ \pi)^{\ast}
   \left( -\psi + d ((\iota_{Z} \psi) \circ J) \right) \\
   &=(\pi_{TN} \circ \pi)^{\ast}
   \left( -\psi + \frac{1}{2}d ((\iota_{2Z} \psi) \circ J) \right) \\
   &=\pi^{\ast} \left( \Theta-\frac{1}{2} d ((\iota_{Z^{M}} \Theta) \circ I_{1}) \right), 
\end{align*}
where $\Theta$ is the two-form given by (\ref{a1}). 
Since $d \psi=0$ and $\iota_{Z} \psi=-d \mu$, it holds $L_{Z^{M}} \Theta=0$. 
The function $f_{1}=f-(1/2)\Theta(Z^{M},I_{1}Z^{M})$
does nowhere vanish by $\mu^{-1}(-c/2) =\emptyset$. 
Therefore Theorem \ref{gene_sup_cmap} leads to the conclusion. 
\end{proof}

Therefore a conical special complex manifold $(N,J,\nabla,\xi)$ with 
a $J$-hermitian, $\nabla$-parallel two-form $\psi$ such that 
$(1/2 \pi) [\psi] \in H^{2}_{DR}(N,\mathbb{Z})$ and $\mu=(1/2) \psi(\xi,J \xi)$ is not surjective  
gives rise to a quaternionic manifold of dimension $2 \dim N$
under a suitable choice of the constant $c$.

For $t \in \mathbb{R}/\pi \mathbb{Z}$, 
we define a connection $\nabla^{t}$ by
$\nabla^{t}=e^{tJ} \circ \nabla \circ e^{-tJ}$, 
which is a special complex connection 
by \cite[Proposition 1]{ACD}. 
Moreover, by 
\begin{equation*}\label{nablat:eq}\nabla^{t}=\nabla - (\sin t)
e^{tJ}(\nabla J)\end{equation*} 
(\cite[Lemma 1]{ACD}), 
we see that $\nabla^{t}$ satisfies $\nabla^{t} \xi=\mathrm{id}$. 
Therefore $\{ \nabla^{t} \}_{t \in \mathbb{R}/\pi \mathbb{Z}}$
is a family of conical special complex connections if $\nabla J \neq 0$. 

\begin{lemma}\label{nabla_t_para}
If $\psi$ is $J$-hermitian and  $\nabla$-parallel,  then 
$\psi$ is $\nabla^{t}$-parallel. 
\end{lemma}
 
\begin{proof}
Since 
$\nabla^t-\nabla$ is a linear combination of $\nabla J$ and $J(\nabla J)=-(\nabla J)J$,  
it suffices to remark that $\nabla\psi=0$, $J\cdot \psi=0$ and, hence, $(\nabla_X J)\cdot \psi=0$ for all $X$. 
Here the dot stands for the action on the tensor algebra by derivations. 
\end{proof} 
 
Hence, Corollary \ref{with_two_form} and Lemma \ref{nabla_t_para} 
imply 

\begin{corollary}\label{two_form_onepara}
If $A(=\nabla J) \neq 0$,
there exists 
an $(\mathbb{R}/\pi \mathbb{Z})$-family 
of quaternionic manifolds obtained from 
a conical special complex manifold with $\psi$ 
under the same assumptions 
of Corollary \ref{with_two_form} by the H/Q-correspondence 
{\rm (}for any chosen function $f$ in the construction{\rm )}.
\end{corollary}

\begin{proof}
By Lemma \ref{nabla_t_para}, $\nabla^{t}_{X} \psi=0$. 
Since $(N,J,\nabla^{t},\xi)$ are conical special complex manifolds, 
we have the conclusion. 
\end{proof}

To give an example, we recall 
the (local) characterization of 
a conical special complex manifold \cite{ACD}. 
Let $(\mathbb{C}^{n+1}, J)$ be the standard complex vector space and 
$U$ an open subset in $\mathbb{C}^{n+1}$ with the standard coordinate system $(z_{0},\dots,z_{n})$. 
We consider a holomorphic one-form $\alpha =\sum F_{i} dz_{i}$ on $U$, 
which is also viewed as a holomorphic map $\phi = \phi_\alpha$ 
from $U$ to $(T^{\ast}U=U \times \mathbb{C}^{n+1} \subset)$ $\mathbb{C}^{2(n+1)}$. 
If $\mathrm{Re} \, \phi : U \to \mathbb{R}^{2(n+1)}$ is an immersion, 
which is equivalent to $\phi$ being totally complex \cite{ACD}, then 
we can find an affine connection $\nabla$ such that $(U,J,\nabla)$ is 
a special complex manifold. In fact, we can take a local coordinate system 
\[ (x_{0}:=\mathrm{Re} \, z_{0},\dots, x_{n}:=\mathrm{Re} \, z_{n},
y_{0}:=\mathrm{Re} \, F_{0},\dots,y_{n}:=\mathrm{Re} \, F_{n}) \]
on $U$ induced by $\phi$ and a connection $\nabla$ defined by the condition that  
$(x_{0},\dots,x_{n},y_{0},\dots,y_{n})$ is affine. 
Moreover $\sum_{i=0}^{n} dx_{i} \wedge dy_{i}$ is $\nabla$-parallel symplectic form on $U$.  
In particular, if $\alpha=-\sum_{i=0}^{n} \sqrt{-1} z_{i} dz_{i}$, 
then the induced affine coordinate system coincides with the real coordinate system underlying the 
holomorphic coordinate system $(z_{0},\dots,z_{n})$, 
hence $(U,J,\nabla)$ is trivial ($\nabla J=0$) in that special case.  
In addition to being holomorphic and totally complex, we assume that 
$\phi$ is conical, which is equivalent to the condition that 
functions $F_{0}, \dots,F_{n}$ are homogeneous of degree one, i.e.\ $F_{i}(\lambda z)=\lambda F_{i}(z)$  
for all $\lambda$ near $1 \in \mathbb{C}^{\ast}$ and $z \in U$. 
Then $U$ is conical, that is, 
any conical holomorphic one-form 
$\phi$ such that $\mathrm{Re} \, \phi$ is an immersion 
on $U$ defines a conical special complex (and symplectic) 
manifold structure of complex dimension $n$. 
Conversely, any such manifold can be locally obtained in this way 
(see \cite[Corollary 5]{ACD}).

If we choose $\alpha=-\sum_{i=0}^{n} \sqrt{-1} z^{i} dz^{i}$ on $\mathbb{C}^{n+1} \backslash \{ 0 \}$, 
then the generalized c-map associates an open submanifold of  
$(\mathbb{H}^{n+1},Q)$ with the standard quaternionic structure $Q$ 
to the complex projective space 
$(\mathbb{C}P^{n},J^{st},[\nabla^{FS}])$, 
where $J^{st}$ is the standard complex structure and $\nabla^{FS}$ is the
Levi-Civita connection of the Fubini-Study metric.
Here we have chosen $\Theta=0$. 
We can also apply 
Corollary \ref{with_two_form} by 
choosing the standard symplectic form as $\psi$. 
More generally, we have the following example.

\begin{example}\label{ex_conic}
{\rm
For a holomorphic function $g$ of homogeneous degree one,  
we consider the holomorphic $1$-form
\[ \alpha =g dz_{0} - \sqrt{-1} \sum_{i=1}^{n} z_{i} dz_{i}  \]
on $U:=\{ (z_{0},z_{1},\dots,z_{n}) \in \mathbb{C}^{n+1} \mid \mathrm{Im}\, g_{0} \neq 0\}$, 
where $g_{i}=\frac{\partial g}{\partial z_{i}}$ $(i=0,1,\dots,n)$. 
{[}Comment Vicente: we should perhaps use a different symbol for $F$ to avoid 
Note that $d \alpha \neq 0$ if there exists $i$ such that $g_{i} \neq 0$ $(i \geq 1)$. 
Setting $z_{i}=u_{i}+\sqrt{-1} v_{i}$ $(i=0,1,\dots,n)$, we have 
\begin{align*}
(x_{0},\dots, x_{n},y_{0},y_{1},\dots,y_{n})
&=\mathrm{Re}\,  \phi (u_{0},\dots, u_{n},v_{0},\dots, v_{n})\\
&=(\mathrm{Re}\,z_{0}, \dots, \mathrm{Re}\,z_{1}, \mathrm{Re}\,g, 
\mathrm{Re}\, (-\sqrt{-1}z_{1}),\dots, \mathrm{Re}\, (-\sqrt{-1}z_{n})) \\
&=(u_{0},\dots, u_{n},\mathrm{Re}\,g,v_{1}, \dots, v_{n} ).
\end{align*}
Since its Jacobian matrix  is given by 
\[   \frac{\partial(x_0, \dots, y_n)}{\partial (u_0,\dots, v_n)} = \left( \begin{array}{cccccccc}
                           1 &            &  &0  & 0 & \dots&\dots & 0\\
                             &  \ddots &  & \vdots &\vdots & & &\vdots \\
                             &            & 1 &0  & 0 & \dots&\dots & 0\\   
    \mathrm{Re}\,g_{0} &  \dots  & \mathrm{Re}\,g_{n} &  -\mathrm{Im}\,g_{0} &  -\mathrm{Im}\,g_{1} & \dots & \dots & -\mathrm{Im}\,g_{n} \\
             0               &  \dots  & 0 & 0 & 1& 0 & \dots & 0 \\ 
             \vdots                 &            &  \vdots & \vdots & 0& 1& & 0 \\ 
             \vdots                 &           &  \vdots &  \vdots  & \vdots &  & \ddots &  \\ 
             0               &  \dots  & 0 & 0 & 0&  & & 1 \\ 
                    \end{array} 
             \right), 
\]
we see that $\mathrm{Re}\, \, \phi$ is an immersion and we obtain a conical special complex structure on $U$. 
The coordinate vector fields of $(x_{0},\dots, y_{n})$ are given by 
\begin{align*}
\frac{\partial}{\partial x_{i}} &=
\frac{\partial}{\partial u_{i}} +  \frac{\mathrm{Re}\,  g_{i}}{\mathrm{Im}\,  g_{0}} 
\frac{\partial}{\partial v_{0}} \,\, (i \geq  0), \\ 
\frac{\partial}{\partial y_{0}} &=
- \frac{1}{\mathrm{Im}\,  g_{0}} \frac{\partial}{\partial v_{0}}, \,\,\, 
\frac{\partial}{\partial y_{j}} 
=
- \frac{\mathrm{Im}\,  g_{j}}{\mathrm{Im}\,  g_{0}} \frac{\partial}{\partial v_{0}}
+\frac{\partial}{\partial v_{j}} \,\, (j  \geq 1)
\end{align*}
on $U$. 
Let $\nabla$ (resp.\ $\nabla^{\text{st}})$ be the flat affine connection on $U$ such that 
$(x_{0},\dots,y_{n})$ (resp. $(u_{0},\dots,v_{n})$)
is a $\nabla$ (resp.\ $\nabla^{\text{st}}$)-affine coordinate system. 
We define $S$ by $\nabla=\nabla^{\text{st}} +S$. Then we calculate
\begin{align*}
0 &= \nabla_{X} \frac{\partial}{\partial x_{i}} = (\nabla^{\text{st}}_{X} +S_{X}) 
\left( \frac{\partial}{\partial u_{i}} + \frac{\mathrm{Re}\,  g_{i}}{\mathrm{Im}\,  g_{0}} 
\frac{\partial}{\partial v_{0}} \right) \\
   &= X \left(\frac{\mathrm{Re}\,  g_{i}}{\mathrm{Im}\,  g_{0}} \right) \frac{\partial}{\partial v_{0}}
      +S_{X} \frac{\partial}{\partial u_{i}} 
      + \frac{\mathrm{Re}\,  g_{i}}{\mathrm{Im}\,  g_{0}} S_{X} \frac{\partial}{\partial v_{0}} \,\, (i \geq 0)
\end{align*}
and similarly we have
\begin{align*}
& -X \left(\frac{1}{\mathrm{Im}\,  g_{0}} \right) \frac{\partial}{\partial v_{0}}
    -\frac{1}{\mathrm{Im}\,  g_{0}} S_{X} \frac{\partial}{\partial v_{0}} = 0, \\
& -X \left(\frac{\mathrm{Im}\,  g_{j}}{\mathrm{Im}\,  g_{0}} \right) \frac{\partial}{\partial v_{0}}
      - \frac{\mathrm{Im}\,  g_{j}}{\mathrm{Im}\,  g_{0}} S_{X} \frac{\partial}{\partial v_{0}} 
      +S_{X} \frac{\partial}{\partial v_{j}} 
= 0 \,\,\, (j>0). 
\end{align*}
From these equations, it holds that
\begin{align}
S_{X} \frac{\partial}{\partial u_{i}} 
&= - \frac{X\mathrm{Re}\, g_{i}}{\mathrm{Im}\,  g_{0}} \frac{\partial}{\partial v_{0}}, \,\, 
S_{X} \frac{\partial}{\partial v_{i}} 
= \frac{X\mathrm{Im}\, g_{i}}{\mathrm{Im}\,  g_{0}} \frac{\partial}{\partial v_{0}} \label{SX_1}
\,\, (i \geq 0).
\end{align}
Using $A_{X}Y=(\nabla_{X} J)(Y)=S_{X}JY-JS_{X}Y$ and (\ref{SX_1}), we have the matrix representation 
\begin{align}\label{nab_J_mat}
A=\nabla J
=\frac{1}{\mathrm{Im} \, g_{0}}
\left( \begin{array}{ccc}
                    A_{0 } & \dots & A_{n} \\
                    0_{2}  &  \cdots & 0_{2}  \\
                    \vdots &  \ddots &  \vdots\\
                    0_{2}  &  \cdots & 0_{2} 
                    \end{array} 
             \right)
\end{align}
of 
$A$ with respect to the frame 
\[ \left( \frac{\partial}{\partial u_{0}},  
\frac{\partial}{\partial v_{0}}, \dots ,
\frac{\partial}{\partial u_{n}},   
\frac{\partial}{\partial v_{n}} \right), \]
where
\[
A_{i}=
\left( \begin{array}{cc}
                    -d \mathrm{Re}\, g_{i} & d \mathrm{Im}\, g_{i} \\
                    d \mathrm{Im}\, g_{i} & d \mathrm{Re}\, g_{i}  
                    \end{array} 
             \right)\,\,\, \mbox{and} \,\,\,\, 
      0_{2}=
\left( \begin{array}{cc}
                    0 & 0 \\
                    0 & 0   
                    \end{array} \right).                      
\]
Note that we change the order of the frame for simplicity. 
This means that $A \neq 0$ if there exists $i$ such that $g_{i} \neq \mbox{constant}$. 
By Lemma \ref{square_B} and (\ref{nab_J_mat}), $A^{2}=(\nabla J)^{2}$ induces a globally defined tensor on $\bar{U}$, 
in particular 
\[ \mathrm{Tr}A^{2} = \mathrm{Tr}A_{0}^{2}
=\frac{2}{(\mathrm{Im}\, g_{0})^{2}}  
(d \mathrm{Re}\, g_{0} \otimes d \mathrm{Re}\, g_{0} 
+d \mathrm{Im}\, g_{0} \otimes d \mathrm{Im}\, g_{0}) 
\] 
also induces the 
the symmetric tensor $\bar{\cal B}$ on $\bar{U}$. 
By Lemma \ref{curvature_p} and (\ref{nab_J_mat}), we see that 
\begin{align*}
R^{\nabla^{\prime}}
&=- \frac{1}{4} A \wedge A 
=
- \frac{1}{4 (\mathrm{Im}\, g_{0})^{2} }
\left( \begin{array}{cccc}
                    A_{0} \wedge A_{0} &  A_{0} \wedge A_{1}&  \dots & A_{0} \wedge A_{n}  \\
                     0_{2}  &  0_{2}& \cdots & 0_{2}  \\
                     \vdots & \vdots &  \ddots &  \vdots\\
                    0_{2}  &  0_{2}& \cdots & 0_{2} 
                    \end{array} 
             \right)
\end{align*}
as the matrix representation.  

Since 
\begin{align*}
dx_{i} &=du_{i}\,\,\,  (i \geq 0), \,\,\,  dy_{0}=\sum_{i=0}^{n}\mathrm{Re}\, g_{i} \, du_{i}
-{\mathrm{Im} \, g_{i}\,} dv_{i}, \\
dy_{j} &=dv_{j} \,\,\ (j >0), 
\end{align*}
a $2$-form $\psi= \sum_{i=1}^{n} dx_{i} \wedge dy_{i}(=\sum_{i=1}^{n} du_{i} \wedge dv_{i})$ is 
$J$-hermitian and  $\nabla$-parallel. Note that 
$\sum_{i=0}^{n} dx_{i} \wedge dy_{i}=dx_{0} \wedge dy_{0} + \psi$ is not $J$-hermitian, 
that is, $(U,J,\nabla, dx_{0} \wedge dy_{0} + \psi)$ is not special K{\"a}hlerian 
if there exists $i>0$ such that $\mathrm{Re}\, g_{i} \neq 0$. 
However it is a special symplectic manifold. 
In fact, it holds 
\[  (dx_{0} \wedge dy_{0}) (\frac{\partial}{\partial u_{0}}, \frac{\partial}{\partial u_{i}})
   =\mathrm{Re}\, g_{i} \,\,\, \mbox{and}\,\,\,  
   (dx_{0} \wedge dy_{0}) (J \frac{\partial}{\partial u_{0}}, J \frac{\partial}{\partial u_{i}})=0. \]
Moreover since 
\begin{align*}
\xi &=\sum_{i=0}^{n} x_{i} \frac{\partial}{\partial x_{i}}
        +y_{i} \frac{\partial}{\partial y_{i}} 
        = \cdots + \sum_{i=1} ^{n} u_{i} \frac{\partial}{\partial u_{i}} 
           + v_{i} \frac{\partial}{\partial v_{i}}, \\
J \xi& =  \cdots + \sum_{i=1}^{n} u_{i} \frac{\partial}{\partial v_{i}} 
           - v_{i} \frac{\partial}{\partial u_{i}},             
\end{align*}
we have $\mu=\psi(\xi,J \xi)=(1/2) \sum_{i=1}^{n} (u_{i}^{2}+v_{i}^{2})$.  
Take a $\mathrm{U}(1)$-bundle $\pi:TU \times \mathrm{U}(1) \to TU$ with a connection form 
\[ \eta=(\pi_{TU} \circ \pi)^{\ast}(\sum_{i=1}^{n}u_{i} dv_{i})+d \theta, \] 
where $\theta$ is the 
angular coordinate of 
$\mathrm{U}(1)$. 
The special case Corollary \ref{with_two_form} of Theorem \ref{gene_sup_cmap}  
can be applied and then we obtain a quternionic manifold. 

We consider the horizontal subbundle of $p_{U}:U \to \bar{U}$ given by the kernel of 
$\kappa=-(1/2s) d \mu \circ J$ on each level set $\mu^{-1}(s) \subset U$ ($s \neq 0 $). 
We retake $U$ as an open set in $\cup_{s>0} \mu^{-1}(s)$. 
For horizontal vector fields $X$ and $Y$ tangent to each level set $\mu^{-1}(s)$, 
$XY \mu =0$ means that 
\[ (p_{U}^{\ast} \bar{a})(X,Y)= \frac{1}{2s} \psi(JX,Y), \]
where $\bar{a}$ is the $\xi$-component of the fundamental tensor of 
$\mathcal{A}^{\nabla^{\prime}}$ as in Section \ref{c-proj:sec} . 
Here we used $d \kappa=\psi/s$. This means that $\bar{a}$ is symmetric and 
$\bar{J}$-hermitian, and hence the Ricci tensor  
of the connection $\bar{\nabla}^{\prime \kappa}$ on $\bar{U}$ induced from $\kappa$ is symmetric and $\bar{J}$-hermitian. 
Therefore it holds 
\begin{align*}
p_{U}^{\ast} \bar{a}
=- \frac{1}{\sum_{i=1}^{n} (u_{i}^{2}+v_{i}^{2})} \sum_{i=1}^{n} d u_{i} \otimes d u_{i}
+  d v_{i} \otimes d v_{i}. 
\end{align*}
Hence the Ricci tensor $Ric^{\bar{\nabla}^{\prime \kappa}}$ of $\bar{\nabla}^{\prime \kappa}$ satisfies
\begin{align*}
&- \frac{1}{\sum_{i=1}^{n} (u_{i}^{2}+v_{i}^{2})} \sum_{i=1}^{n} d u_{i} \otimes d u_{i}
+  d v_{i} \otimes d v_{i} \\
&=\frac{1}{4(n+1) (\mathrm{Im}\, g_{0})^{2}} (d \mathrm{Re}\, g_{0} \otimes d \mathrm{Re}\, g_{0} 
+d \mathrm{Im}\, g_{0} \otimes d \mathrm{Im}\, g_{0}) 
-\frac{1}{2(n+1)} p_{\bar{U}}^{\ast} (Ric^{\bar{\nabla}^{\prime \kappa}} )              
\end{align*}
by (\ref{bar_a}).
In particular, we see that $Ric^{\bar{\nabla}^{\prime \kappa}} \geq 0$. 
For example, when we choose $g=-\sqrt{-1}z_{1}^{l}/z_{0}^{l-1}$ for $l(\neq 1) \in \mathbb{Z}$, 
we obtain 
\begin{align*}
d \mathrm{Re}\, g_{0} 
&= \frac{\sqrt{-1}}{2}(-l+1)l(-w^{l-1}dw+\bar{w}^{l-1}d \bar{w}), \\
d \mathrm{Im}\, g_{0} 
&= -\frac{1}{2}(-l+1)l(w^{l-1}dw+\bar{w}^{l-1}d \bar{w}), \\
d \mathrm{Re}\, g_{1} 
&= \frac{\sqrt{-1}}{2}(-l+1)l(w^{l-2}dw-\bar{w}^{l-1}d \bar{w}), \\
d \mathrm{Im}\, g_{1} 
&= \frac{1}{2}(-l+1)l(w^{l-2}dw+\bar{w}^{l-2}d \bar{w}), \\
d \mathrm{Re}\, g_{j} 
&=d \mathrm{Im}\, g_{j} =0 \,\,\ (j>1),  
\end{align*}
where $w=z_{1}/z_{0}$. 
%
We denote the corresponding objects with subscript $l$ for ones given by 
$g=-\sqrt{-1}z_{1}^{l}/z_{0}^{l-1}$. It holds that 
\begin{align}
R^{\nabla^{l \, \prime}}
&=-\frac{\sqrt{-1} \, l^{2} \, | w |^{2(l-2)}}{(w^{l}+\bar{w}^{l})^{2}}
\left( \begin{array}{ccccccc}
                    0 & -|w|^{2} & -\mathrm{Im}\, w & \mathrm{Re} w & 0 & \dots & 0\\
                    |w|^{2} & 0   & -\mathrm{Re}\, w & -\mathrm{Im} w & 0 & \dots & 0 \\
                    0 &  0 &  0 & 0  & 0 & \dots & 0\\
                    \vdots &  \vdots &  \vdots  & \vdots  & \vdots & \ddots & \vdots\\
                    0 &  0 &  0 & 0  & 0 & \dots & 0\\
                    \end{array} 
             \right) dw \wedge d \bar{w} \label{c_l}
\end{align}
and 
\begin{align}
\mathrm{Tr}(A^{l})^{2}=
\mathrm{Tr}(\nabla^{l} J)^{2}    
=\frac{4 \, l^{2} \, | w |^{2(l-1)}}{ (w^{l}+\bar{w}^{l})^{2}}
(dw \otimes d \bar{w}+d \bar{w} \otimes dw). \label{b_l}
\end{align}
%
%
%
%
%
%

Finally we consider the quaternionic Weyl curvature 
of $TU$.  
Let $W^{q}$ be the quaternionic Weyl curvature of the quaternionic structure 
$Q=\langle I_{1}, I_{2}, I_{3} \rangle$. 
In \cite{AM}, the explicit expression of $W^{q}$ is given
and it is shown that $W^{q}$ is independent of the choice of 
the quaternionic connection. 
Since the Obata connection of the c-map is Ricci flat by Theorem~\ref{ricci_flat},  
we have $W^{q, l}=R^{\tilde{\nabla}^{0,l}}$ for 
$g=-\sqrt{-1}z_{1}^{l}/z_{0}^{l-1}$. 
If $l \neq 1$, then we see that 
\[ W^{q, l}_{X^{v},Y^{v}} Z^{v} 
=R^{{\tilde{\nabla}^{0,l}}}_{X^{v},Y^{v}} Z^{v} 
=\left (
R^{\nabla^{l \, \prime}}
_{X,Y} Z \right)^{v}.  
\]
Because the vertical lift is determined by a differential manifold structure (not by a connection), 
we see that $W^{q, l} \neq W^{q, k}$ 
on $T(U_{k} \cap U_{l})$ if $l \neq k$, where 
$U_{j}=\{ (z_{0},z_{1},\dots,z_{n}) \in \mathbb{C}^{n+1} \mid \mathrm{Im}\, g_{0} \neq 0\}
=\{ (z_{0},z_{1},\dots,z_{n}) \in \mathbb{C}^{n+1} \mid \mathrm{Re}\, (z_{1}/z_{0})^{j} \neq 0\}
$ for 
$g=-\sqrt{-1}z_{1}^{j}/z_{0}^{j-1}$.  
Here we used (\ref{c_l}). 
So we can find different quaternionic structures 
$Q^{\alpha_{1}}, \dots, Q^{\alpha_{t}}$ on 
$T(\bigcap_{i=1}^{t} U_{\alpha_{i}})$, where $1\neq \alpha_{i} \in \mathbb{Z}$. 
Note that $Q^{0}$ is the flat quaternionic structure. 
}
\end{example}

\begin{remark}
{\rm
Since $d \alpha \neq 0$ except the trivial case $g=-\sqrt{-1}z_{0}$, 
Example \ref{ex_conic} with $g=-\sqrt{-1}z_{1}^{l}/z_{0}^{l-1}$ ($l \neq 0$), 
which is local one,  is not given by a local 
special K{\"a}hlerian one. 
}
\end{remark}

\begin{remark}
{\rm 
For a conical special K{\"a}hler manifold $N$, the particular twist data which yields the quaternionic K{\"a}hler structure of the supergravity c-map on $T^{\ast}N \cong TN$ is 
given in \cite[Lemma 5.1]{MS} in consistency with \cite{ACDM}.  
As we noted in the introduction, we also have a freedom in the choice of the data $\Theta$ etc.\  
for our generalized supergravity c-map.  
For instance, the two form $\Theta$ can be chosen as trivial ($\Theta=0$) or as in equation (\ref{a1}). 
For illustration, we can give yet another possible choice of $\Theta$. 
Assume that $\dim N \geq 6$.  
\label{appendix:sec}
Let $\{ \bar{U}_{\alpha} \}_{\alpha \in \Lambda}$ be an open covering of
$\bar{N}$ with local trivializations 
$U_{\alpha}:=p_{N}^{-1}(\bar{U}_{\alpha}) \cong \bar{U}_{\alpha} \times \mathbb{C}^{\ast}$ 
and $g_{\alpha \beta} : \bar{U}_{\alpha} \cap \bar{U}_{\beta} \to \mathbb{C}^{\ast}$ 
be the corresponding transition functions. 
Let $(r_{\alpha},\theta_{\alpha})$ be the polar coordinates with respect to a 
(smooth) local trivialization 
$p_{N}^{-1}(\bar{U}_{\alpha}) \cong \bar{U}_{\alpha} \times \mathbb{C}^*$ for each 
$\alpha \in \Lambda$. 
A principal connection $\eta$ is locally given by 
\[ \eta = p_{N}^{\ast}(\gamma^{\alpha}_{1} \otimes 1+ \gamma^{\alpha}_{2} \otimes \sqrt{-1}) 
+ \left(\frac{dr_{\alpha}}{r_{\alpha}} \otimes 1 +d \theta_{\alpha} \otimes \sqrt{-1} \right) \]
for a $\mathbb{C}$-valued one-form 
$\gamma^{\alpha}_{1} \otimes 1+ \gamma^{\alpha}_{2} \otimes \sqrt{-1}$
on $\bar{U}_{\alpha} \subset \bar{N}$ for each $\alpha \in \Lambda$. 
If we write 
$g_{\alpha \beta}=e^{f^{1}_{\alpha \beta} + f^{2}_{\alpha \beta} \sqrt{-1}}$, then 
\begin{align*}
&f^{1}_{\alpha \beta} + f^{1}_{\beta \gamma} - f^{1}_{\alpha \gamma} =0, \\ 
&f^{2}_{\alpha \beta} + f^{2}_{\beta \gamma} - f^{2}_{\alpha \gamma} \in 2 \pi \mathbb{Z}, \\
&\gamma^{1}_{\beta}-\gamma^{1}_{\alpha} =d f^{1}_{\alpha \beta}, \\
&\gamma^{2}_{\beta}-\gamma^{2}_{\alpha} =d f^{2}_{\alpha \beta}. 
\end{align*}
Therefore we obtain a principal $\mathrm{U}(1)$-bundle $p_{S}:S \to \bar{N}$ 
with transition functions $e^{f^{2}_{\alpha \beta} \sqrt{-1}}:\bar{U}_{\alpha} \cap \bar{U}_{\beta} \to \mathrm{U}(1)$ and 
connection $\eta_{S}$ locally given by 
\[ p_{S}^{\ast}(\gamma^{\alpha}_{2} \otimes \sqrt{-1})
+ d \theta_{\alpha} \otimes \sqrt{-1}.  \]
In fact, the collection $\{ e^{f^{2}_{\alpha \beta} \sqrt{-1}} \}$ of local 
$\mathrm{U}(1)$-valued functions
satisfies the cocycle condtion 
and the collection $\{ \gamma_{\alpha} \}$ of local 
$\sqrt{-1}\mathbb{R}$-valued 
one-forms  
satisfying 
$\gamma^{2}_{\beta}-\gamma^{2}_{\alpha} =d f^{2}_{\alpha \beta}$ defines a connection 
form $\eta_{S}$. 
By Proposition \ref{curv} and (\ref{a_2}),  
its curvature $d \eta_{S} (=p_{S}^{\ast}(d \gamma_{2}^{\alpha}))$ is $2(\bar{a}_{\bar{J}})^{a}$, where 
$(\bar{a}_{\bar{J}})^{a}$ is given by 
\[ 
(\bar{a}_{\bar{J}})^{a} = \frac{1}{8(n+1)} \bar{\cal B}_{\bar{J}}
-\frac{1}{2} (P^{\bar{\nabla}^{\prime}}_{\bar{J}})^{a}. 
\]
On $TN$, we choose the two-form 
$\Theta=2 (p_{N} \circ \pi_{TN})^{\ast}((\bar{a}_{\bar{J}})^{a} )$ and 
consider the pull-back connection $(p_{N}{}_{\#} \circ \pi_{TN}{}_{\#})^{\ast} \eta_{S}$
on the pull-back bundle $P={\pi_{TN}}^{\#} p_{N}^{\#}S$. 
Since $\iota_{Z^{M}}\Theta=0$, we can see that the assumptions in 
Theorem \ref{gene_sup_cmap} hold. It is left for future studies  
to find a canonical choice of $\Theta$ for the generalized supergravity c-map, which 
allows to invert the H/Q-correspondence of \cite{CH}. 
}
\end{remark}

As an application of Theorem \ref{gene_sup_cmap}, we have the following corollary
by patching 
quaternionic manifolds locally constructed by the generalized supergravity c-maps.

\begin{corollary}
Let $(M,J,[\nabla])$ be a complex manifold with a c-projective structure 
$[\nabla]$ and $\dim M=2n$. If $2n=\dim M \geq 4$ 
and the harmonic curvature of its normal Cartan connection 
vanishes, then there exists a $4(n+1)$-dimensional   
quaternionic manifold $(\check{M},Q)$ 
with the vanishing quaternionic Weyl curvature  
foliated by $(n+2)$-dimensional complex manifolds whose leaves space is $M$.
\end{corollary}

\begin{proof}
Since 
$\dim M \geq 4$ and the harmonic curvature of its normal Cartan connection 
vanishes, $(M,J,[\nabla])$ is locally isomorphic to $(\mathbb{C}P^{n},J^{st},[\nabla^{FS}])$
(see \cite{CEMN} for example). 
So we may assume that $M=\bigcup_{\alpha} U_{\alpha}$, where 
$U_{\alpha}$ is an open subset $\mathbb{C}P^{n}$. 
Set $V_{\alpha}:=p^{-1} (U_{\alpha})$, where 
$p:\mathbb{C}^{n+1} \backslash \{ 0 \} \to \mathbb{C}P^{n}$ is the projection.  
We consider the standard complex structure and 
the standard flat connection induced from $\mathbb{C}^{n+1}$
on each $V_{\alpha}$. 
By Theorem \ref{gene_sup_cmap},  
we have a quaternionic manifold $W_{\alpha}:=\varphi^{\prime}(TV_{\alpha}) 
\subset \mathbb{H}^{n+1}$, where $\varphi^{\prime}$
is the diffeomorphism given in  
Example \ref{ex_conical_hyp}.  
Here we have chosen the two-form $\Theta=0$  and $f=f_{1}=1$ on 
$TV_{\alpha}$ for each $\alpha$. 
We set $\check{M}:=\bigcup_{\alpha} W_{\alpha}$.
The induced quaternionic structure on each $W_{\alpha}$ coincides with 
the standard one from $\mathbb{H}^{n+1}$. 
Hence an almost quaternion structure $Q$ on $\check{M}$ 
can be obtained. 
Since there exists a quaternionic connection on each $W_{\alpha}$, one can obtain 
a quaternionic connection on $\check{M}$ by the partition of unity, that is, $Q$ is a
quaternionic structure with vanishing quaternionic Weyl curvature. 
For each $p \in TV_{\alpha} \cap TV_{\beta}$, the leaf of $\cal{L}$ through $p$ in $TV_{\alpha}$ is 
denoted by $L^{\alpha}$ and corresponding leaf in $W_{\alpha}$ is denoted by $\hat{L}^{\alpha}$, 
that is $\hat{L}^{\alpha}=\varphi^{\prime}(L^{\alpha})$. 
Since $\hat{L}^{\alpha}=\hat{L}^{\beta}$ in $\check{M}$, 
we obtain leaves in $\check{M}$ and see that its leaves space is $M$. 
Since the subbundle 
$\mathcal{L}$ is an $I_{1}$-invariant in $T(TV_{\alpha})$, 
each leaf $L$ is a complex manifold with $I:=I_{1}|_{L}$. 
Each leaf $\hat{L}$ on $\check{M}$ is obtained by the Swann's twist with 
an almost complex structure $\hat{I}$. 
By \cite[Proposition 3.8]{Sw} and $\Theta=0$, $\hat{I}$ is integrable.  
\end{proof}
%
%

\noindent
{\bf Acknowledgments.} 
We thank Aleksandra Bor\'owka for comments. 
Research by the first author is partially funded by the Deutsche Forschungsgemeinschaft (DFG, German Research 
Foundation) under Germany's
Excellence Strategy -- EXC 2121 Quantum Universe -- 390833306. 
The second author's research is partially supported by 
JSPS KAKENHI Grant Number 18K03272.

\noindent
Vicente Cort{\' e}s \\
Department of Mathematics \\
and Center for Mathematical Physics \\
University of Hamburg \\
Bundesstra\ss e 55, \\
D-20146 Hamburg, Germany. \\ 
email:vicente.cortes@uni-hamburg.de \\

\noindent
Kazuyuki Hasegawa \\
Faculty of teacher education \\
Institute of human and social sciences \\
Kanazawa university \\
Kakuma-machi, Kanazawa, \\
Ishikawa, 920-1192, Japan. \\
e-mail:kazuhase@staff.kanazawa-u.ac.jp

\end{document}